\documentclass[leqno,b5paper]{amsart}

\usepackage{tikz}
\usepackage{tikz-cd}
\usepackage{blindtext}
\usetikzlibrary{automata, arrows.meta, positioning, decorations.markings}
\usepackage{amsmath,amsfonts,amsthm,latexsym}
\usepackage{mathrsfs, textcomp}
\newtheorem{theorem}{Theorem}[section]
\newtheorem{lemma}[theorem]{Lemma}
\newtheorem{proposition}[theorem]{Proposition}
\newtheorem{corollary}[theorem]{Corollary}

\theoremstyle{definition}
\newtheorem{definition}[theorem]{Definition}
\newtheorem{example}[theorem]{Example}
\newtheorem{obs}[theorem]{Observation}

\numberwithin{equation}{section}

\usepackage{amssymb}
\usepackage{amscd}
\usepackage{eqlist, color}
\usepackage{mathrsfs, textcomp}
\usepackage{color, eqlist}
\usepackage{graphicx}
\usepackage{bm}
\usepackage{amsmath,amsfonts,amsthm,latexsym}
\usepackage[all]{xy}
\usepackage{tikz}
\usepackage{tikz-cd}

\DeclareMathOperator{\id}{id}

\DeclareMathOperator{\inte}{int}

\DeclareMathOperator{\Rmt}{Rem}
\DeclareMathOperator{\Rem}{\textbf{Rem}}

\DeclareMathOperator{\Int}{int}

\DeclareMathOperator{\cl}{cl}
\DeclareMathOperator{\Morp}{Morp}
\DeclareMathOperator{\Obj}{Obj}

\DeclareMathOperator{\BCFLoc}{\textbf{BiCFLoc}}

\DeclareMathOperator{\BBooLoc}{\textbf{BiBooLoc}}

\DeclareMathOperator{\BiLoc}{\textbf{BiLoc}}
\DeclareMathOperator{\RBiLoc}{\textbf{RemBiLoc}}
\DeclareMathOperator{\Loc}{\textbf{Loc}}
\DeclareMathOperator{\Frm}{\textbf{Frm}}
\DeclareMathOperator{\BiFrm}{\textbf{BiFrm}}

\theoremstyle{definition}\newtheorem{thm}{Theorem}[section]
\theoremstyle{definition}
\theoremstyle{definition}
\theoremstyle{definition}
\theoremstyle{definition}
\theoremstyle{remark}\newtheorem{rem}[thm]{Remark}
\theoremstyle{definition}

\theoremstyle{definition}

\begin{document}
	\title[$(i,j)$-remote sublocales]
	{Remoteness in the category of bilocales}

	\author{Mbekezeli Nxumalo}
	\address{Department of Mathematical Sciences, University of South Africa, P.O. Box 392, 0003 Pretoria, SOUTH AFRICA.}
	\address{Department of Mathematics, Rhodes University, P.O. Box 94, Grahamstown 6140, South Africa.}
	\email{sibahlezwide@gmail.com}
	\dedicatory{}
	
	
	\let\thefootnote\relax\footnote{}
	
	\begin{abstract} In locale theory, a sublocale is said to be remote in case it misses every nowhere dense sublocale. In this paper, we introduce and study a new class of sublocales in the category of bilocales, namely $(i,j)$-remote sublocales. These are bilocalic counterparts of remote sublocales and are the sublocales missing every $(i,j)$-nowhere dense sublocale, with  $(i,j)$-nowhere dense sublocales being bilocalic counterparts of $(\tau_{i},\tau_{j})$-nowhere dense subsets in bitopological spaces. A comprehensive study of $(i,j)$-nowhere dense sublocales is given and we show that in the class of balanced bilocales, a sublocale is $(i,j)$-nowhere dense if and only if its bilocale closure is nowhere dense. We also consider weakly $(i,j)$-remote sublocales which are those sublocales missing every clopen $(i,j)$-nowhere dense sublocale. Furthermore, we extend $(i,j)$-remoteness to the categories of bitopological spaces as well as normed lattices. In the class of $\sup$-T$_{D}$ bitopological spaces, a subset $A$ of a bitopological space $(X,\tau_{1},\tau_{2})$ is $(\tau_{i},\tau_{j})$-remote if and only if the induced sublocale $\widetilde{A}$ of $\tau_{1}\vee\tau_{2}$ is $(i,j)$-remote.  Given a bilocale $(L,L_{1},L_{2})$, the collection $\Rmt_{(i,j)}\!L$ of all elements of $L$ inducing closed weakly $(i,j)$-remote sublocales is a closed sublocale of $L$ and is always $(i,j)$-remote but seldomly remote. For the congruence bilocale of a locale $L$, $\Rmt_{(i,j)}\!\mathfrak{C}L=\mathfrak{C}L$ and for the ideal bilocale of a bilocale whose total part is Noetherian, $\Rmt_{(i,j)}L=L$ if and only if $\Rmt_{(i,j)}\mathfrak{J}L=\mathfrak{J}L$.  We show that $\Rmt_{(i,j)}$ is a functor  from the category $\BiLoc_{R(i,j)}$ of bilocales whose morphisms are $\Rmt_{(i,j)}$-maps to the category of locales and there is a natural transformation from $\Rmt_{(i,j)}$ to the functor $G$ which is the composite of the faithful functor $F:\BiLoc\rightarrow \Loc$ and the inclusion functor $\BiLoc_{R(i,j)}\hookrightarrow\BiLoc$. Moreover, there is a comonad associated with the endofunctor from the category $\RBiLoc_{SR(i,j)}$ whose objects are symmetric bilocales in which $\Rmt_{(i,j)}\!L$ is remote. We prove that the category of symmetric Boolean bilocales is a coreflective subcategory of the category $\RBiLoc_{SR(i,j)}$.
	\end{abstract}
	
	\maketitle
	
	
	\section*{Introduction}\label{sect0}
	
	In \cite{N}, we introduced and studied remote sublocales in the category of locales. These are sublocales that miss every nowhere dense sublocale. Remote sublocales subsequently appeared in the article \cite{N2} where their relationship with maximal nowhere dense sublocales was investigated. The aim of this paper is to study remoteness in the category of bilocales. We shall define and characterize $(i,j)$-remote sublocales and weakly $(i,j)$-remote sublocales of bilocales. These sublocales are defined using the notion of an $(i,j)$-nowhere dense sublocale which is a bilocale counterpart of a \textit{$(\tau_{i},\tau_{j})$-nowhere dense} subset which was  defined by Fukutake \cite{F}, in 1992, as a subset $A$ of a bitopological space (bispace in short) $(X,\tau_{1},\tau_{2})$ such that $\Int_{\tau_{j}}(\cl_{\tau_{i}}(A))=\emptyset$ $(i\neq j)$, where $\Int_{\tau_{i}}$ and $\cl_{\tau_{i}}$ (for $i=1,2$) denote the $\tau_{i}$-interior and $\tau_{i}$-closure, respectively. $(i,j)$-remote sublocales are also extended to the categories of bitopological spaces and normed lattices and we prove that in the class of $\sup$-T$_{D}$ bitopological spaces, a subset $A$ of a bitopological space $(X,\tau_{1},\tau_{2})$ is $(\tau_{i},\tau_{j})$-remote if and only if the induced sublocale $\widetilde{A}$ of $\tau_{1}\vee\tau_{2}$ is $(i,j)$-remote. We commented in \cite{N} that we could not determine whether the collection of all elements inducing closed remote sublocales is always a sublocale. In this paper, we prove that $\Rmt_{(i,j)}\!L$ is always a closed sublocale. The assignment $\Rmt_{(i,j)}$ is a functor from the category $\BiLoc_{(i,j)}$ of bilocales whose morphisms are $\Rmt_{(i,j)}$-maps to the category of locales and there is a natural transformation from $\Rmt_{(i,j)}$ to the functor $G$ which is the composite of the faithful functor $F:\BiLoc\rightarrow \Loc$ and the inclusion functor $\BiLoc_{R(i,j)}\hookrightarrow\BiLoc$. 
	
	This paper is organized as follows. Section one consists of the necessary background. In Section two we introduce $(i,j)$-nowhere dense sublocales and study their properties that will be used in the other sections. We prove that, although $(i,j)$-nowhere dense sublocales cannot always be characterized using the Booleanization of the total part of a bilocale, in the category of balanced bilocales, a sublocale is $(i,j)$-nowhere dense precisely when its $i$-closure misses the Booleanization of the total part of the locale. In Section three, we define and investigate $(i,j)$-remote sublocales as well as weakly $(i,j)$-remote sublocales. We show among other things that in the category of symmetric bilocales, $(i,j)$-remote sublocales are those of the form $\mathfrak{b}(x^{*})$. We extend $(i,j)$-remoteness to the categories of bitopological spaces and normed lattices. It turns out that, in a category of Sup-T$_{D}$ bispaces, a subset $A$ of a bitopological space $(X,\tau_{1},\tau_{2})$ is $(\tau_{i},\tau_{j})$-remote if and only if $\widetilde{A}$ is $(i,j)$-remote as a sublocale of the bilocale $(\tau_{1}\vee\tau_{2},\tau_{1},\tau_{2})$. The fourth section discusses preservation and reflection of $(i,j)$-remote sublocales by bilocalic maps. In Section five, we show that in a bilocale $(L,L_{1},L_{2})$, the collection $\Rmt_{(i,j)}\!L$ of elements of $L$ inducing closed weakly $(i,j)$-remote sublocales is a closed sublocale, which is not always remote. We also prove that $\Rmt_{(i,j)}\!L$ coincides with $L$ exactly when $L$ is $(i,j)$-remote as a sublocale of itself. For the congruence bilocale of a locale $L$, $\Rmt_{(i,j)}\!\mathfrak{C}L=\mathfrak{C}L$ and for the ideal bilocale of a bilocale whose total part is Noetherian, $\Rmt_{(i,j)}L=L$ if and only if $\Rmt_{(i,j)}\mathfrak{J}L=\mathfrak{J}L$. We prove that $\Rmt_{(i,j)}$ is a functor  from the category $\BiLoc_{R(i,j)}$ of bilocales whose morphisms are $\Rmt_{(i,j)}$-maps to the category of locales and also show among other things  that the category of symmetric Boolean bilocales is a coreflective subcategory of the category $\RBiLoc_{SR(i,j)}$ whose objects are symmetric bilocales in which $\Rmt_{(i,j)}\!L$ is remote.
	
	\section{Preliminaries}\label{sect1}
	See \cite{PP1} as reference for notions of locales and  \cite{BBH,PP2} for the theory of bilocales.
	\subsection{Locales }\label{subsect11}
	
	A \textit{locale} $L$ is a complete lattice satisfying: 
	
	$$a\wedge\bigvee B=\bigvee\{a\wedge b:b\in B\}$$ for all $a\in L$, $B\subseteq L$. 
	$1_{L}$ and $0_{L}$, with subscripts dropped if there is no possibility of confusion, respectively denote the top element and the bottom element of a locale $L$. We denote by $a^{*}$ the \textit{pseudocomplement} of an element $a\in L$. An element $a\in L$ is said to be \textit{dense } and \textit{complemented} in case $a^{*}=0$ and $a\vee a^{\ast}=1$, respectively. 
	
	We denote by $\mathfrak{O}X$ the locale of open subsets of a topological space $X$. 
	
	A \textit{localic map} is an infima-preserving function $f:L\rightarrow M$ between locales such that the corresponding left adjoint $f^{*}$, called its associated \textit{frame homomorphism}, preserves binary meets.  Throughout the paper, if $f$ is a localic map, we will write $h$ for its associated frame homomorphism.
	
	$\Frm$ and $\Loc$ represent the categories of locales whose morphisms are frame homomorphisms and localic maps, respectively.
	
	
	\subsection{Sublocales}\label{subsect12}
	A \textit{sublocale} of a locale $L$ is a subset $S$ which is closed under all meets and $x\rightarrow s\in S$ for every $x\in L$ and $s\in S$, where $\rightarrow$ is a \textit{Heyting operation} on $L$ satisfying that $a\leq b\rightarrow c$ if and only if $a\wedge b\leq c$ for all $a,b,c\in L$. We write $\mathcal{S}(L)$ for the collection of all sublocales of a locale $L$. We denote by $\mathsf{O}$ the smallest sublocale of a locale $L$ and we shall say that a sublocale $S$ misses $T\in\mathcal{S}(L)$ in case $T\cap S=\mathsf{O}$. The largest sublocale of a locale $L$ missing a sublocale $S$ of $L$ is denoted by $L\smallsetminus S$. The sublocales
	$${\mathfrak{c}}(a)=\{x\in L:a\leq x\}\quad\text{and}\quad \mathfrak{o}(a)=\{a\rightarrow x:x\in L\},$$ of a locale $L$ are respectively the \textit{closed} and \textit{open} sublocales induced by an element $a$ of $L$. They are complements of each other. The smallest closed sublocale of $L$ containing $S\in \mathcal{S}(L)$ is called the \textit{closure} of $S$ and denoted by $\overline{S}$. $S\in \mathcal{S}(L)$ is dense and nowhere dense if $\overline{S}=L$ and $S\cap\mathfrak{B}L=\mathsf{O}$, respectively, where 
	$\mathfrak{B}(L)=\{x\rightarrow 0:x\in L\}$ is the least dense sublocale of $L$. 
	
	
	For each sublocale $S\subseteq L$ there is an onto frame homomorphism $\nu_{S}:L\rightarrow S$ defined by $\nu_{S}(a)={\bigwedge}\{s\in S: a\leq s\}.$ Open sublocales and closed sublocales of a sublocale $S$ of $L$ are given by $$\mathfrak{o}_{S}(\nu_{S}(a))=S\cap \mathfrak{o}(a)\quad\text{and}\quad {\mathfrak{c}_{S}}(\nu_{S}(a))=S\cap \mathfrak{c}(a),$$ respectively, for $a\in L$. For any sublocale $S$ of a locale $L$ and $x\in L$, $S\subseteq\mathfrak{o}(x)$ if and only if $\nu_{S}(x)=1$.
	
	Each localic map $f:L\rightarrow M$ induces the functions $f[-]:\mathcal{S}(L)\rightarrow \mathcal{S}(M)$ given by the set-theoretic image of each sublocale of $L$ under $f$, and   $f_{-1}[-]:\mathcal{S}(M)\rightarrow\mathcal{S}(L)$ given by $$f_{-1}[T]={\bigvee}\{A\in \mathcal{S}(L):A\subseteq f^{-1}(T)\}.$$ For a localic map $f:L\rightarrow M$ and $x\in M$,  $$f_{-1}[\mathfrak{c}_{M}(x)]=\mathfrak{c}_{L}(h(x))\quad\text{and}\quad f_{-1}[\mathfrak{o}_{M}(x)]=\mathfrak{o}_{L}(h(x)).$$
	
	We denote by $\widetilde{A}$ a sublocale of $\mathfrak{O}X$ \textit{induced} by a subset $A$ of a topological space $X$. 
	\subsection{Bilocales}
	
	A \textit{bilocale} is a triple $(L,L_{1}, L_{2})$ where $L_{1},L_{2}$ are subframes of a locale $L$ and for all $a\in L$, \[ a=\bigvee\{a_{1}\wedge a_{2}: a_{1}\in L_{1}, a_{2}\in L_{2}\text{ and } a_{1}\wedge a_{2}\leq a\}.\] 
	
	We call $L$ the \textit{total part} of $(L,L_{1}, L_{2})$, and $L_{1}$ and $L_{2}$ the first and second parts, respectively. We use the notations $L_{i},L_{j}$ to denote the first or second parts of $(L,L_{1},L_{2})$, always assuming that $i,j=1,2$, $i\neq j$. 

	The \textit{bilocale pseudocomplement} of $c\in L_{i}$ is given by $$c^{\bullet}=\bigvee\{x\in L_{j}:x\wedge c=0\}.$$


	A \textit{biframe homomorphism} (or \textit{biframe map}) $h:(M,M_{1},M_{2})\rightarrow(L,L_{1},L_{2})$ is a frame homomorphism $h:M\rightarrow L$ for which $h(M_{i})\subseteq L_{i}\quad (i=1,2).$
	
	We write $\BiFrm$ for the category of bilocales whose morphisms are biframe maps.

	\section{$(i,j)$-nowhere dense sublocales}
	We devote this section to introducing $(i,j)$-nowhere dense sublocales from $(i,j)$-nowhere dense subspaces and studying some of their properties. 
	
	Recall from \cite{F} that for a bitopological space (bispace in short) $(X,\tau_{1},\tau_{2})$, where $\Int_{\tau_{i}}$ and $\cl_{\tau_{i}}$ (for $i=1,2$) denote the $\tau_{i}$-interior and $\tau_{i}$-closure, respectively, a  subset $A$ of $X$ is \textit{$(\tau_{i},\tau_{j})$-nowhere dense} in $X$ if $\Int_{\tau_{j}}(\cl_{\tau_{i}}(A))=\emptyset$ $(i\neq j)$. We will extend this notion into locales and explore some of its bilocalic properties.
	
	For a bilocalic notion of $(\tau_{i},\tau_{j})$-nowhere density, we introduce bilocalic counterparts of the notions of closure and interior. Let $(L,L_{1},L_{2})$ be a bilocale. In \cite{PP2}, the authors introduced the following notation for a sublocale $S\subseteq L$: $$\cl_{i}(S)=\bigcap\{\mathfrak{c}(a):a\in L_{i}, S\subseteq\mathfrak{c}(a)\}=\mathfrak{c}\left(\bigvee\{a\in L_{i}:S\subseteq\mathfrak{c}(a)\}\right)\quad(i=1,2).$$
	
	We define $\Int_{i}(S)$ as follows:
	$$\Int_{i}(S)=\bigvee\{\mathfrak{o}(a):a\in L_{i}, \mathfrak{o}(a)\subseteq S\}\quad(i=1,2).$$
	
	We shall refer to these concepts as \textit{bilocale closure} and \textit{bilocale interior}, respectively. Throughout this paper, we assume that $i\neq j\in\{1,2\}$, unless otherwise stated.

	
	We define an $(i,j)$-nowhere dense sublocale as follows.
	
	\begin{definition}\label{ijnowheredense}
		Let $(L,L_{1},L_{2})$ be a bilocale. A sublocale $S$ of $L$ is \textit{$(i,j)$-nowhere dense} if $\Int_{j}(\cl_{i}(S))=\mathsf{O}$ $(i\neq j\in\{1,2\})$.
	\end{definition}
	
	Our discussion of $(i,j)$-nowhere density involves bilocale interiors and bilocale closures. Before we consider the properties of bilocale closures and bilocale interiors which will be useful below, we collect some properties of the bilocale pseudocomplement in the following proposition. We refer the reader to \cite{N1} for the proofs of the following two results.
	
	\begin{proposition}\label{pseudo}
		Let $(L,L_{1},L_{2})$ be a bilocale and $i\neq j\in\{1,2\}$. Then
		\begin{enumerate}
			\item $0^{\bullet}=1$.
			\item For every $a\in L_{i}$, $a\wedge a^{\bullet}=0$.
			\item $a\wedge b=0$ iff $a\leq b^{\bullet}$ for all $a\in L_{j},b\in L_{i}$.
			\item $a\leq b$ implies $b^{\bullet}\leq a^{\bullet}$ for all $a,b\in L_{i}$.
			\item For each $a\in L_{i}$, $a\leq a^{\bullet\bullet}$.
			\item For each $a\in L_{i}$, $a^{\bullet}=a^{\bullet\bullet\bullet}$.
			\item $(a\vee b)^{\bullet}=a^{\bullet}\wedge b^{\bullet}$ for every $a,b\in L_{i}$.
		\end{enumerate}
	\end{proposition}
	
	We gather in one result some properties of bilocale closure and bilocale interior. 
	
	\begin{proposition}\label{int}
		Let $(L,L_{1},L_{2})$ be a bilocale and $S,T\in\mathcal{S}(L)$. The following statements hold for $i\neq j\in\{1,2\}$.
		\begin{enumerate}
			\item \cite{PP2} $S\subseteq\overline{ S}\subseteq\cl_{i}(S)$.
			\item If $T\subseteq S$, then $\cl_{i}(T)\subseteq \cl_{i}(S)$.
			\item $\cl_{i}(\cl_{i}(S))=\cl_{i}(S)$.
			\item $\mathfrak{c}(a)=\cl_{i}(\mathfrak{c}(a))$ for every $a\in L_{i}$.
			\item $\Int_{i}(S)=\mathfrak{o}\left(\bigvee\{a\in L_{i}:\mathfrak{o}(a)\subseteq S\}\right)$.
			\item $\Int_{i}(S)\subseteq\Int(S)\subseteq S$.
			\item If $T\subseteq S$, then $\Int_{i}(T)\subseteq\Int_{i}(S)$.
			\item $\Int_{i}(\Int_{i}(S))=\Int_{i}(S)$.
			\item $\mathfrak{o}(a)=\Int_{i}(\mathfrak{o}(a))$ for every $a\in L_{i}$.
			\item For each $a\in L_{i}$, $\mathfrak{c}(a^{\bullet})=\cl_{j}(\mathfrak{o}(a))$.
			\item For each $a\in L_{i}$, $\mathfrak{o}(a^{\bullet})=\Int_{j}(\mathfrak{c}(a))$.
			\item For each $a\in L_{i}$, $\cl_{j}(\mathfrak{o}(a))=L\smallsetminus\Int_{j}(\mathfrak{c}(a))$.
			\item For each $a\in L_{i}$, $\Int_{j}(\mathfrak{c}(a))=L\smallsetminus\cl_{j}(\mathfrak{o}(a))$.
			\item For each $a\in L$, $L\smallsetminus\Int_{i}(\mathfrak{o}(a))=\cl_{i}(\mathfrak{c}(a))$.
			\item For each $a\in L$, $L\smallsetminus\cl_{i}(\mathfrak{c}(a))=\Int_{i}(\mathfrak{o}(a))$.
		\end{enumerate}
	\end{proposition}

	In what follows, we introduce $i$-dense sublocales. Recall from \cite{KS} that a subset $A$ of a bitopological space $(X,\tau_{1},\tau_{2})$ is \textit{$i$-dense} if $\cl_{\tau_{i}}(A)=X$. This recalled notion motivates the following definition of an $i$-dense sublocale.
	
	\begin{definition}\label{idensesub}
		A sublocale $A$ of a bilocale $(L,L_{1},L_{2})$ is \textit{$i$-dense} if $\cl_{i}(A)=L$.
	\end{definition}
	
	We work towards showing that the bilocalic definition of $i$-density is ``conservative" in the sense that a subset $A$ of a bitopological space $(X,\tau_{1},\tau_{2})$ is $i$-dense if and only if $\widetilde{A}$ is $i$-dense in $(\tau_{1}\vee\tau_{2},\tau_{1},\tau_{2})$.
	
	Recall from \cite{L} that given a topological property $P$, a bitopological space $(X,\tau_{1},\tau_{2})$ is $\sup$-$P$ if $(X,\tau_{1}\vee\tau_{2})$ has property $P$. We say that $(X,\tau_{1},\tau_{2})$ is $\sup$-T$_{D}$ if $(X,\tau_{1}\vee\tau_{2})$ is T$_{D}$.
	
	In bilocalic terms, we denote the sublocale induced by a subset $A$ of $X$ as follows:$$\widetilde{A}=\{\Int_{\tau_{1}\vee\tau_{2}}((X\smallsetminus A)\cup G):G\in\tau_{1}\vee\tau_{2}\}.$$This notation has all the properties of usual induced sublocales. We shall denote by $\tau$ the topology $\tau_{1}\vee\tau_{2}$.
	
	We have the following:

	\begin{quote}
		\emph{In a $\sup\text{-T}_{D}$ bitopological space $(X,\tau_{1},\tau_{2})$, $x\in A\subseteq X$ if and only if $\widetilde{x}\in\widetilde{A}$.}
	\end{quote}
	
	Lemma \ref{biclo} below provides a useful property of bilocale closure. The proof is similar to that of \cite[Proposition 2.10.]{N} and shall be omitted.
	
	\begin{lemma}\label{biclo}
		Let $A$ be a subset of a $\sup$-T$_{D}$-bispace $(X,\tau_{1},\tau_{2})$. Then $\widetilde{\cl_{\tau_{i}}(A)}=\cl_{i}(\widetilde{A})$ for $i=1,2$.
	\end{lemma}
	
	As a result of Lemma \ref{biclo}, we have the following proposition which shows that in the class of $\sup$-T$_{D}$-bispaces, the definition of $i$-density given in Definition \ref{idensesub} is conservative in bilocales.
	\begin{proposition}
		Let $A$ be a subset of a $\sup$-T$_{D}$-bispace $(X,\tau_{1},\tau_{2})$. Then $A$ is $i$-dense iff $\widetilde{A}$ is $i$-dense.
	\end{proposition}
	\begin{proof}
		A subset $A$ of $X$ is $i$-dense if and only if $\cl_{\tau_{i}}(A)=X$ if and only if $\widetilde{\cl_{\tau_{i}}(A)}=\widetilde{X}$ if and only if $\cl_{i}(\widetilde{A})=\widetilde{X}$ if and only if $\widetilde{A}$ is $i$-dense.
	\end{proof}
	
	We give an elementary notion of $i$-density.
	
	\begin{definition}\label{idense}
		Define an element $x\in L_{j}$ of a bilocale $(L,L_{1},L_{2})$ to be \textit{$L_{i}$-dense} (or just \textit{$i$-dense}) if $x^{\bullet}=0$. 
	\end{definition}
	
	The following result gives a characterization of $i$-dense elements.
	\begin{proposition}\label{bullet}
		Let $(L,L_{1},L_{2})$ be a bilocale and $x\in L_{j}$. Then the following statements are equivalent.
		\begin{enumerate}
			\item $x$ is $i$-dense.
			\item $\mathfrak{o}(x)$ is $i$-dense.
			\item For all $a\in L_{i}$, $a\wedge x=0$ implies $a=0$.
		\end{enumerate} 
	\end{proposition}
	\begin{proof}
		$(1)\Leftrightarrow(2)$: Observe that for any $x\in L_{j}$,	\begin{align*}
			x^{\bullet}=0
			&\quad \Leftrightarrow\quad
			\mathfrak{c}(x^{\bullet})=L\\
			&\quad \Leftrightarrow\quad
			\cl_{i}(\mathfrak{o}(x))=L\quad\text{since}\quad \mathfrak{c}(x^{\bullet})=\cl_{i}(\mathfrak{o}(x))\ \text{from Proposition \ref{int}(10)}\\
			&\quad \Leftrightarrow\quad
			\mathfrak{o}(x) \quad\text{is}\quad j\text{-dense}.
		\end{align*}
		$(2)\Rightarrow(3)$: If $a\in L_{i}$ such that $x\wedge a=0$, then $\mathfrak{o}(x)\subseteq\mathfrak{c}(a)$ which implies that $$L=\mathfrak{c}(0)=\cl_{i}(\mathfrak{o}(x))\subseteq\cl_{i}(\mathfrak{c}(a))=\mathfrak{c}(a).$$ Thus $a=0$.

		$(3)\Rightarrow(1)$: Recall that $x^{\bullet}\wedge x=0$ by Proposition \ref{pseudo}(2). The hypothesis gives $x^{\bullet}=0$. Thus $x$ is $i$-dense.
	\end{proof}
	
	\begin{obs}\label{idenseobs}
		It is easy to see that an element $x\in L_{j}$ that is dense in $L$ is $i$-dense.
	\end{obs}
	
	
	In the next result, we show that in the class of $\sup$-T$_{D}$-bispaces, the definition of $(i,j)$-nowhere density given in Definition \ref{ijnowheredense} is conservative in bilocales. Recall that for each $U\in \tau_{i}$, $$U^{\bullet}=\bigvee\{G\in\tau_{j}:G\cap U=\emptyset\}=\bigvee\{G\in\tau_{j}:G\subseteq X\smallsetminus U\}=\inte_{\tau_{j}}(X\smallsetminus U)=X\smallsetminus\cl_{\tau_{j}}(U).$$
	
	\begin{proposition}\label{ijndinduced}
		Let $(X,\tau_{1},\tau_{2})$ be a $\sup$-T$_{D}$-bispace. A subset $A\subseteq X$ is $(\tau_{i},\tau_{j})$-nowhere dense iff $\widetilde{A}$ is $(i,j)$-nowhere dense.
	\end{proposition}
	\begin{proof}
		Observe that 
		\begin{align*}
			\Int_{\tau_{j}}(\cl_{\tau_{i}}(A))=\emptyset
			&\quad \Leftrightarrow\quad
			X\smallsetminus\Int_{\tau_{j}}(\cl_{\tau_{i}}(A))=X\\
			&\quad \Leftrightarrow\quad
			X\smallsetminus\cl_{\tau_{j}}(X\smallsetminus\cl_{\tau_{i}}(A))=\emptyset\\
			&\quad \Leftrightarrow\quad
			(X\smallsetminus\cl_{\tau_{i}}(A))^{\bullet}=\emptyset\quad\text{since}\quad U^{\bullet}=X\smallsetminus \cl_{\tau_{j}}(U)\quad\text{for all}\quad U\in\tau_{i}\\
			&\quad \Leftrightarrow\quad
			\mathfrak{o}((X\smallsetminus\cl_{\tau_{i}}(A))^{\bullet})=\mathsf{O}\\
			&\quad \Leftrightarrow\quad
			\Int_{j}(\mathfrak{c}(X\smallsetminus\cl_{\tau_{i}}(A)))=\mathsf{O}\quad\text{from Proposition \ref{int}(9)}\\
			&\quad \Leftrightarrow\quad
			\Int_{j}\left(\widetilde{\cl_{\tau_{i}}(A)}\right)=\mathsf{O}\quad\text{since}\quad\cl_{\tau_{i}}(A)\quad\text{is}\quad \tau\text{-closed} \\
			&\quad \Leftrightarrow\quad
			\Int_{j}\left(\cl_{i}(\widetilde{A})\right)=\mathsf{O}\quad\text{since}\quad \widetilde{\cl_{\tau_{i}}(A)}=\cl_{i}(\widetilde{A})
		\end{align*}which proves the result.
	\end{proof}
	
	We characterize $(i,j)$-nowhere dense sublocales in Theorem \ref{ijnd} below. We shall need the following lemma whose proof is straighforward.
	
	\begin{lemma}\label{ijndsubset}
		Let $(L,L_{1},L_{2})$ be a bilocale with $S, T\in\mathcal{S}(L)$. If $S$ is $(i,j)$-nowhere dense and $T\subseteq S$, then $T$ is $(i,j)$-nowhere dense.
	\end{lemma}
	
	\begin{theorem}\label{ijnd}
		Let $(L,L_{1},L_{2})$ be a bilocale and $S\in\mathcal{S}(L)$. The following statements are equivalent.
		\begin{enumerate}\label{ndbilocale}
			\item $S$ is $(i,j)$-nowhere dense.
			\item $L\smallsetminus\cl_{i}(S)$ is $j$-dense. 
			\item $\left(\bigvee\{x\in L_{i}:S\subseteq\mathfrak{c}(x)\}\right)^{\bullet}=0$.
			\item $\bigvee\{x\in L_{i}:S\subseteq\mathfrak{c}(x)\}$ is $j$-dense.
			\item $\cl_{i}(S)$ is $(i,j)$-nowhere dense.
			\item $\overline{S}$ is $(i,j)$-nowhere dense.
		\end{enumerate}
	\end{theorem}
	\begin{proof}
		$(1)\Leftrightarrow(2)$: We have that \begin{align*}
			\Int_{j}(\cl_{i}(S))=\mathsf{O}&\Leftrightarrow\mathfrak{o}\left(\bigvee\{a\in L_{j}:\mathfrak{o}(a)\subseteq\cl_{i}(S)\}\right)=\mathsf{O}\\
			&\Leftrightarrow \mathfrak{c}\left(\bigvee\{a\in L_{j}:\mathfrak{o}(a)\subseteq\cl_{i}(S)\}\right)=L\\
			&\Leftrightarrow \mathfrak{c}\left(\bigvee\{a\in L_{j}: (L\smallsetminus\cl_{i}(S))\subseteq \mathfrak{c}(a)\}\right)=L\quad\text{since }\cl_{i}(S)\text{ is complemented}.\\
			&\Leftrightarrow \cl_{j}(L\smallsetminus \cl_{i}(S))=L\quad\text{by Proposition \ref{int}(4)}.
		\end{align*}
		
		$(2)\Leftrightarrow(3)$: Observe that \begin{align*}
			\cl_{j}(L\smallsetminus\cl_{i}(S))=L
			&\quad \Leftrightarrow\quad
			\cl_{j}\left(L\smallsetminus\mathfrak{c}\left(\bigvee\{x\in L_{i}:S\subseteq\mathfrak{c}(x)\}\right)\right)=L\\
			&\quad \Leftrightarrow\quad
			\cl_{j}\left(\mathfrak{o}\left(\bigvee\{x\in L_{i}:S\subseteq\mathfrak{c}(x)\}\right)\right)=L\\
			&\quad \Leftrightarrow\quad
			\mathfrak{c}\left(\left(\bigvee\{x\in L_{i}:S\subseteq\mathfrak{c}(x)\}\right)^{\bullet}\right)=L\quad\text{by Proposition \ref{int}(10)}\\
			&\quad \Leftrightarrow\quad
			\left(\bigvee\{x\in L_{i}:S\subseteq\mathfrak{c}(x)\}\right)^{\bullet}=0.
		\end{align*}

		$(3)\Leftrightarrow(4)$: Follows from definition of $j$-density.

		$(4)\Leftrightarrow(5)$: We have that
		\begin{align*}
			\bigvee\{x\in L_{i}:S\subseteq\mathfrak{c}(x)\}\text{ is } j\text{-dense}
			&\Leftrightarrow
			\bigvee\{x\in L_{i}:\cl_{i}(S)\subseteq\cl_{i}(\mathfrak{c}(x))=\mathfrak{c}(x)\}\text{ is } j\text{-dense}\\
			& \Leftrightarrow
			\bigvee\{x\in L_{i}:\mathfrak{o}(x)\subseteq L\smallsetminus\cl_{i}(S)\}\text{ is } j\text{-dense}\\
			&\Leftrightarrow
			\cl_{j}\left(\mathfrak{o}\left(\bigvee\{x\in L_{i}:\mathfrak{o}(x)\subseteq L\smallsetminus\cl_{i}(S)\}\right)\right)=L \\
			&\Leftrightarrow
			\cl_{j}\left(\mathfrak{o}\left(\bigvee\{x\in L_{i}:S\subseteq\mathfrak{c}(x)\}\right)\right)=L \\
			&\Leftrightarrow
			\cl_{j}\left(L\smallsetminus\mathfrak{c}\left(\bigvee\{x\in L_{i}:S\subseteq\mathfrak{c}(x)\}\right)\right)=L \\
			&\Leftrightarrow
			\cl_{j}\left(L\smallsetminus \cl_{i}(S)\right)=L\\
			&\Leftrightarrow
			L\smallsetminus\cl_{j}\left(L\smallsetminus \cl_{i}(S)\right)=\mathsf{O} \\
			&\Leftrightarrow\mathfrak{o}\left(\bigvee\{a\in L_{j}:L\smallsetminus \cl_{i}(S)\subseteq\mathfrak{c}(a)\}\right)=\mathsf{O}\\
			&\Leftrightarrow\mathfrak{o}\left(\bigvee\{a\in L_{j}:\mathfrak{o}(a)\subseteq \cl_{i}(S)\}\right)=\mathsf{O}\\
			& \Leftrightarrow
			\Int_{j}\big(\cl_{i}(S)\big)=\mathsf{O}\\
			&\Leftrightarrow
			\Int_{j}\big(\cl_{i}(\cl_{i}(S))\big)=\mathsf{O}\quad\text{since } \cl_{i}(\cl_{i}(S))=\cl_{i}(S)\\
			&\Leftrightarrow
			\cl_{i}(S)\ \text{is}\ (i,j)\text{-nowhere dense}.
		\end{align*}
		
		$(5)\Rightarrow(6)$: Since $\overline{S}\subseteq\cl_{i}(S)$ and $\cl_{i}(S)$ is $(i,j)$-nowhere dense, it follows from Lemma \ref{ijndsubset} that $\overline{S}$ is $(i,j)$-nowhere dense.
		
		$(6)\Rightarrow(1)$: This is another application of Lemma \ref{ijndsubset}.
	\end{proof}
	
	In terms of closed sublocales, we get the following characterization of $(i,j)$-nowhere dense sublocales.
	
	\begin{corollary}\label{nd elements in bilocale}
		An element $a\in L_{i}$ is $j$-dense iff $\mathfrak{c}(a)$ is $(i,j)$-nowhere dense.
	\end{corollary}

\begin{rem}
	For any bilocale $(L,L_{1},L_{2})$ and every sublocale $S$ of $L$, $L\smallsetminus\cl_{i}(S)$ is clopen if and only if $\bigvee\{x\in L_{i}:\mathfrak{o}(x)\subseteq S\}$ is complemented in $L$: Follows since $L\smallsetminus\cl_{i}(S)$ is an open sublocale of $L$. 
\end{rem}

As a result of the preceding lemma, we have the following result about clopen $(i,j)$-nowhere dense sublocales. The proof is similar to that of Theorem \ref{ijnd} and shall be omitted.

\begin{proposition}Let $(L,L_{1},L_{2})$ be a bilocale and $S\in\mathcal{S}(L)$. The following statements are equivalent:
	\begin{enumerate}\label{ndbilocaleclopen}
		\item $S$ is clopen $(i,j)$-nowhere dense.
		\item $L\smallsetminus\cl_{i}(S)$ is clopen $j$-dense. 
		\item $\left(\bigvee\{x\in L_{i}:S\subseteq\mathfrak{c}(x)\}\right)^{\bullet}=0$ and $\bigvee\{x\in L_{i}:S\subseteq\mathfrak{c}(x)\}$ is complemented in $L$.
		\item $\bigvee\{x\in L_{i}:S\subseteq\mathfrak{c}(x)\}$ is $j$-dense and complemented in $L$.
		\item $\cl_{i}(S)$ is clopen $(i,j)$-nowhere dense.
		\item $\overline{S}$ is clopen $(i,j)$-nowhere dense.
	\end{enumerate}
\end{proposition}

\begin{corollary}
	An element $a\in L_{i}$ is $j$-dense and complemented in $L$ iff $\mathfrak{c}(a)$ is clopen $(i,j)$-nowhere dense.
\end{corollary}

In locales, $\mathsf{O}$ is the only clopen nowhere dense sublocale of a locale. However, in bilocales, this need not be the case, as shown below.

	\begin{example}\label{BLnotnd}
		Recall that the void sublocale is the only clopen nowhere dense sublocale of a locale. This follows since if $S$ is a clopen and nowhere dense sublocale of a locale $L$, then $$\mathsf{O}=\inte(\overline{S})=\inte(S)=S.$$ In bilocales, for a given bilocale $(L,L_{1},L_{2})$, a sublocale $S$ of $L$ may be a clopen sublocale of $L$ and at the same be $(i,j)$-nowhere dense without being void. For instance, consider the bitopological space $(X,\tau_{1},\tau_{2})$, where $X=\{a,b,c\}$, $\tau_{1}=\{\emptyset,X,\{a\},\{b,c\}\}$ and $\tau_{2}=\{\emptyset,X,\{b\}\}$. We get that $\tau=\{\emptyset,X,\{a\},\{b\},\{b,c\},\{a,b\}\}$. It is clear that $\mathfrak{c}(\{b,c\})$ is a nonvoid clopen $(1,2)$-nowhere dense sublocale of $\tau$. 
	\end{example}
	
	In locales, nowhere dense sublocales are precisely those that miss the least dense sublocale. We have the following result in bilocales.
	
	\begin{proposition}\label{smallest dense in bilocale}
		Let $(L,L_{1},L_{2})$ be a bilocale and $S\in\mathcal{S}(L)$. Then $S$ is $(i,j)$-nowhere dense whenever $\cl_{i}(S)\cap\mathfrak{B}L=\mathsf{O}$.
	\end{proposition}
	\begin{proof}
		Observe that \begin{align*}
			\cl_{i}(S)\cap\mathfrak{B}L=\mathsf{O}
			&\quad \Leftrightarrow\quad
			\mathfrak{B}L\subseteq L\smallsetminus\cl_{i}(S)\text{ since }\cl_{i}(S)\text{ is complemented}\\
			&\quad \Leftrightarrow\quad
			\overline{L\smallsetminus\cl_{i}(S)}=L\\
			&\quad \Rightarrow\quad
			\cl_{j}(L\smallsetminus\cl_{i}(S))=L\quad\text{by Proposition \ref{int}(1)}\\
			&\quad \Leftrightarrow\quad
			\Int_{j}(\cl_{i}(S))=\mathsf{O}
		\end{align*}where the latter equivalence can be deduced from the proof of Proposition \ref{ijnd}$(4)\Leftrightarrow(5)$.
	\end{proof}
	The converse of Proposition \ref{smallest dense in bilocale} is not always true. For instance, using Theorem \ref{ijnd}(4), we see that, in Example \ref{BLnotnd}, the sublocale $\widetilde{\{a\}}=\{\{b,c\},X\}=\mathfrak{c}(\{b,c\})$ of $\tau$ is $(1,2)$-nowhere dense, but $\cl_{1}(\widetilde{\{a\}})=\widetilde{\{a\}}$ does not miss $\mathfrak{B}\tau=\{\emptyset,X,\{a\},\{b,c\}\}$.

	In Proposition \ref{smallest dense in bilocale wrt to Li} below, we improve Proposition \ref{smallest dense in bilocale}.
	
	Recall from \cite{FS} that a bilocale $(L,L_{1},L_{2})$ is
	\textit{balanced} if $x\in L_{1}$ implies $x^{*}\in L_{2}$  and $x\in L_{2}$ implies $x^{*}\in L_{1}$.
	In a balanced bilocale $(L,L_{1},L_{2})$, $a^{*}=a^{\bullet}$ for all $a\in L_{i}$. Indeed, it is clear that $a^{\bullet}\leq a^{*}$. Furthermore, if $y=a^{*}$, then $y\in L_{j}$ and $y\wedge a=0$. Therefore $y\in \{x\in L_{j}:a\wedge x=0\}$. Thus $$y=a^{*}\leq \bigvee\{x\in L_{j}:a\wedge x=0\}=a^{\bullet}.$$ 
	
	\begin{proposition}\label{smallest dense in bilocale wrt to Li}
		Let $(L,L_{1},L_{2})$ be a balanced bilocale and $N\in\mathcal{S}(L)$. Then $N\in\mathcal{S}(L)$ is $(i,j)$-nowhere dense iff $\mathfrak{B}L\cap \cl_{i}(N)=\mathsf{O}$.
	\end{proposition}
	\begin{proof}
		For each $N\in\mathcal{S}(L)$, we have that
		\begin{align*}
			N\quad\text{is}\quad(i,j)\text{-nowhere dense}
			&\Leftrightarrow
			\left(\bigvee\{x\in L_{i}:N\subseteq\mathfrak{c}(x)\}\right)^{\bullet}=0\text{ by Theorem \ref{ijnd}}\\
			&\Leftrightarrow
			\left(\bigvee\{x\in L_{i}:N\subseteq\mathfrak{c}(x)\}\right)^{*}=0\text{ since }(L,L_{1},L_{2})\text{ is balanced}\\
			& \Leftrightarrow
			\mathfrak{o}\left(\bigvee\{x\in L_{i}:N\subseteq\mathfrak{c}(x)\}\right)\ \text{is dense}\\
			& \Leftrightarrow
			\mathfrak{B}L\subseteq\mathfrak{o}\left(\bigvee\{x\in L_{i}:N\subseteq\mathfrak{c}(x)\}\right)\\
			&\Leftrightarrow
			\mathfrak{B}L\cap\mathfrak{c}\left(\bigvee\{x\in L_{i}:N\subseteq\mathfrak{c}(x)\}\right)=\mathsf{O}\\
			&\Leftrightarrow
			\mathfrak{B}L\cap\cl_{i}(N)=\mathsf{O}
		\end{align*}which proves the result.
	\end{proof}
The above result implies the folowing.
\begin{corollary}\label{ijndbalanced}
	In a balanced bilocale $(L,L_{1},L_{2})$, a sublocale $N$ of $L$ is $(i,j)$-nowhere dense iff $\cl_{i}(N)$ is nowhere dense.
\end{corollary}
	\section{$(i,j)$-remote sublocales}
	The aim of this section is to introduce new classes of sublocales called $(i,j)$-remote sublocales and weakly $(i,j)$-remote sublocales and study their properties.
	\begin{definition}\label{remotesubbilocale} 
		Let $(L,L_{1},L_{2})$ be a bilocale. A sublocale $S$ of $L$ is \textit{$(i,j)$-remote (resp. weakly $(i,j)$-remote)} if $N\cap S=\mathsf{O}$ for every $(i,j)$-nowhere dense (resp. clopen $(i,j)$-nowhere dense) sublocale $N$.
	\end{definition}
	
	
	We consider some examples.
	
	\begin{example}\label{exaBLij}
		
		(1) The void sublocale is both $(i,j)$-remote and weakly $(i,j)$-remote.
		
		(2) If $S$ is an $(i,j)$-remote (resp. weakly $(i,j)$-remote) sublocale of $L$, then every sublocale of $L$ contained in $S$ is $(i,j)$-remote (resp. weakly $(i,j)$-remote).
		
		(3) In a \textit{symmetric biframe}, which was defined in \cite{BB} as a biframe $(L,L_{1},L_{2})$ in which $L=L_{1}=L_{2}$ (we shall only write $L$ for a symmetric bilocale whose total part is $L$) we have:
		
		\begin{itemize}
			\item[a)] $(i,j)$-remoteness coincides with remoteness. As a result, the sublocale $\mathfrak{B}L$ is $(i,j)$-remote.
			\item[b)] Every sublocale is weakly $(i,j)$-remote. This follows since $\mathsf{O}$ is the only clopen $(i,j)$-nowhere dense sublocale of the locale $L$.
		\end{itemize} 
		
		(4) Every $(i,j)$-remote sublocale is weakly $(i,j)$-remote, but the converse is not always true: Since, by \cite{N}, Boolean locales are precisely those whose sublocales are remote, it follows from (3)(b) that in a symmetric bilocale $L$ where $L$ is non-Boolean, not every weakly $(i,j)$-remote sublocale is $(i,j)$-remote.
		
		(5) For any symmetric bilocale $L$, $\{\mathfrak b(x^*) : x\in L\}$ is the collection of all $(i,j)$-remote sublocales of $L$: Let $A$ be an $(i,j)$-remote sublocale of $L$. Since $\mathfrak{B}L$ is the largest remote sublocale and remoteness coincides with $(i,j)$-remoteness, $A\subseteq \mathfrak BL$. Because every sublocale of a Boolean locale is Boolean, we have that $A$ is a Boolean sublocale of $L$ and so there is an $a\in L$ such that $A=\mathfrak b(a)$. Since $a\in\mathfrak b(a)$, it follows that $a\in\mathfrak BL$, and so $a=a^{**}$. 
			
			On the other hand, let $x\in L$ and consider the sublocale $\mathfrak b(x^*)$ of $L$. Since $x^*$ belongs to $\mathfrak BL$ and $\mathfrak b(x^*)$ is the smallest sublocale containing $x^*$, we have $\mathfrak b(x^*)\subseteq \mathfrak BL$, and so $\mathfrak b(x^*)$ is remote and hence an  $(i,j)$-remote sublocale of $L$. 
			
				(6) In a balanced bilocale, every remote sublocale is $(i,j)$-remote (In particular, the Booleanization is $(i,j)$-remote): Let $(L,L_{1},L_{2})$ be a balanced bilocale, $S$ a remote sublocale of $L$ and choose an $(i,j)$-nowhere dense sublocale $N$. By Corollary \ref{ijndbalanced}, $\cl_{i}(N)$ is nowhere dense so that $S\cap \cl_{i}(N)=\mathsf{O}$, making $S\cap N=\mathsf{O}$ as required. 
	\end{example}
	
	In the following main result of this section, we give a characterization of $(i,j)$-remote and weakly $(i,j)$-remote sublocales. We only prove results about $(i,j)$-remote sublocales. Results about weakly $(i,j)$-remote sublocales follow a similar sketch.
	\begin{theorem}\label{remote subbilocale}
		Let $(L,L_{1},L_{2})$ be a bilocale and $S\in \mathcal{S}(L)$. The following statements are equivalent.
		\begin{enumerate} 
			\item $S$ is $(i,j)$-remote (resp. weakly $(i,j)$-remote).
			\item $S\cap \cl_{i}(N)=\mathsf{O}$ for every $(i,j)$-nowhere dense (resp. clopen $(i,j)$-nowhere dense) $N$.
			\item $S\cap\overline{N}=\mathsf{O}$ for every $(i,j)$-nowhere dense (resp. clopen $(i,j)$-nowhere dense) $N$.
			\item $S\cap\mathfrak{c}(x)=\mathsf{O}$ for all $j$-dense (resp. complemented $j$-dense) $x\in L_{i}$.
			\item $S\subseteq \mathfrak{o}(a)$ for every $j$-dense (resp. complemented $j$-dense) $a\in L_{i}$.
			\item $\nu_{S}(d)=1$ for every $j$-dense (resp. complemented $j$-dense) $d\in L_{i}$.
		\end{enumerate}
	\end{theorem}
	\begin{proof}
		
		$(1)\Leftrightarrow(2)\Leftrightarrow(3)$: Follows since a sublocale $N$ of $L$ is $(i,j)$-nowhere dense if and only if $\cl_{i}(N)$ is $(i,j)$-nowhere dense if and only if $\overline{N}$ is $(i,j)$-nowhere dense, by Theorem \ref{ijnd}.

		$(3)\Rightarrow(4)$: Let $x\in L_{i}$ be $j$-dense. It follows that $\mathfrak{c}(x)$ is $(i,j)$-nowhere dense. By (3), $$\mathsf{O}=S\cap \overline{\mathfrak{c}(x)}=S\cap\mathfrak{c}(x).$$

		$(4)\Leftrightarrow(5)$: Follows since $S\cap\mathfrak{c}(y)=\mathsf{O}$ if and only if $S\subseteq \mathfrak{o}(y)$ for all $S\in\mathcal{S}(L)$ and every $y\in L$.

		$(5)\Rightarrow(6)$: Let $d\in L_{i}$ be $j$-dense. By (5), $S\subseteq \mathfrak{o}(d)$. Since $\nu_{B}(a)=1$  if and only if $B\subseteq\mathfrak{o}(a)$ for every $a\in L,B\in\mathcal{S}(L)$, $\nu_{S}(d)=1$.
		
		$(6)\Rightarrow(1)$: Let $N\in\mathcal{S}(L)$ be $(i,j)$-nowhere dense. It follows from Theorem \ref{ijnd} that $\bigvee\{a\in L_{i}:N\subseteq \mathfrak{c}(a)\}$ is a $j$-dense element of $L_{i}$. By hypothesis, $\nu_{S}(\bigvee\{a\in L_{i}:N\subseteq \mathfrak{c}(a)\})=1$ which implies that $S\subseteq \mathfrak{o}\left(\bigvee\{a\in L_{i}:N\subseteq \mathfrak{c}(a)\}\right)$. Therefore $$\mathsf{O}=S\cap\mathfrak{c}\left(\bigvee\{a\in L_{i}:N\subseteq \mathfrak{c}(a)\}\right)=S\cap \cl_{i}(N)\supseteq S\cap N$$which proves the implication.
	\end{proof}

\begin{proposition}\label{largestijremote}
	Let $(L,L_{1},L_{2})$ be a bilocale. Then there is the largest $(i,j)$-remote sublocale of $L$.
\end{proposition}
\begin{proof}
	Consider a collection $\{S_{i}:i\in I\}$ of $(i,j)$-remote sublocales of $L$. For any $j$-dense $x\in L_{i}$, we have that 
		$$\mathfrak{c}(x)\cap \bigvee_{i\in I}\{S_{i}:i\in I\}=\bigvee_{i\in I}\{\mathfrak{c}(x)\cap S_{i}:i\in I\}=\bigvee_{i\in I}\{\mathsf{O}\}=\mathsf{O}.$$
\end{proof}
\begin{obs}
	(1) The above result also holds for weakly $(i,j)$-remote sublocales.
	
	(2) 
	Since for all sublocales $S$ and $T$, the sublocale $S\smallsetminus T$ is contained in $S$, and since a sublocale smaller than an $(i,j)$-remote (resp. weakly $(i,j)$-remote) sublocale is $(i,j)$-remote (resp. weakly $(i,j)$-remote), we deduce from (1) above and Proposition \ref{largestijremote} that the system of $(i,j)$-remote (resp. weakly $(i,j)$-remote) sublocales, partially ordered by inclusion, is a coframe.
\end{obs}
	Next, we discuss $(i,j)$-remoteness and weakly $(i,j)$-remoteness of closed sublocales. 
	
	\begin{proposition}\label{bidense}
		Let $(L,L_{1},L_{2})$ be a bilocale and $a\in L$. Then $\mathfrak{c}(a)$ is  $(i,j)$-remote (resp. weakly $(i,j)$-remote) iff $a\vee x=1$ for every $j$-dense (resp. complemented $j$-dense) $x\in L_{i}$.
	\end{proposition}
	\begin{proof}
		For each $j$-dense $x\in L_{i}$, Theorem \ref{remote subbilocale} implies that  \begin{align*}
			\mathfrak{c}(a)\cap \cl_{i}(\mathfrak{c}(x))=\mathsf{O}
			&\quad \Leftrightarrow\quad
			\mathfrak{c}(a)\cap\mathfrak{c}(x)=\mathsf{O}\\
			&\quad \Leftrightarrow\quad
			\mathfrak{c}(a\vee x)=\mathsf{O}\\
			&\quad \Leftrightarrow\quad
			a\vee x=1
		\end{align*}which proves the result.
	\end{proof}
	
	
	We close this section with a short discussion of remoteness in the categories of bitopological spaces and normed lattices. 
	
	We start by introducing $(\tau_{i},\tau_{j})$-remote subsets of bispaces.
	
	\begin{definition}
		Let $(X,\tau_{1},\tau_{2})$ be a bispace and $A\subseteq X$. Then $A$ is $(\tau_{i},\tau_{j})$-remote if $A\cap F=\emptyset$ for every $(\tau_{i},\tau_{j})$-nowhere dense $F\subseteq X$.
	\end{definition}
	
	We show that a subset $A$ of a $\sup$-T$_{D}$ bispace $(X,\tau_{1},\tau_{2})$ is $(\tau_{i},\tau_{j})$-remote (resp. weakly $(\tau_{i},\tau_{j})$-remote) precisely when $\widetilde{A}$ is $(i,j)$-remote (resp. weakly $(\tau_{i},\tau_{j})$-remote). We shall need the following result taken from \cite[Observation 2.1.17.]{N1}:
	
	\begin{quote}
		\emph{For a T$_{D}$-space $X$ and subsets $A$ and $B$ of $X$ where $A$ is either closed or open, $A\cap B=\emptyset$ if and only if $\widetilde{A}\cap\widetilde{B}=\mathsf{O}$. }
	\end{quote}
	
	\begin{proposition}\label{remotebilocaleprop}
		Let $(X,\tau_{1},\tau_{2})$ be a $\sup$-T$_{D}$ bitopological space and $A\subseteq X$. Then $\widetilde{A}$ is $(i,j)$-remote iff $A$ is $(\tau_{i},\tau_{j})$-remote, where $i\neq j\in\{1,2\}$.
	\end{proposition}
	\begin{proof}
		$(\Rightarrow):$ Let $N\subseteq X$ be $(\tau_{i},\tau_{j})$-nowhere dense. By Proposition \ref{ijndinduced}, $\widetilde{N}$ is $(i,j)$-nowhere dense in $\tau$. By hypothesis, $\widetilde{A}\cap\cl_{i}(\widetilde{N})=\mathsf{O}$. But $\cl_{i}(\widetilde{N})=\widetilde{\cl_{\tau_{i}}(N)}$, so $\widetilde{A}\cap\widetilde{\cl_{\tau_{i}}(N)}=\mathsf{O}$. Therefore $S_{(A\cap \cl_{\tau_{i}}(N))}=\mathsf{O}$ making $A\cap \cl_{\tau_{i}}(N)=\emptyset$.

		$(\Leftarrow):$ Let $F$ be an  $(i,j)$-nowhere dense sublocale of $\tau$. By Theorem \ref{ijnd}, $\cl_{i}(F)$ is $(i,j)$-nowhere dense in $\tau$. But $\cl_{i}(F)$ is $\tau$-closed, so there is a $\tau$-closed $B\subseteq X$ such that $\cl_{i}(F)=\widetilde{B}$. It follows from Proposition \ref{ijndinduced} that $B$ is $(\tau_{i},\tau_{j})$-nowhere dense. Therefore $B\cap A=\emptyset$. By \cite[Observation 2.1.17.]{N}, $$\mathsf{O}=\widetilde{B}\cap \widetilde{A}=\cl_{i}(F)\cap \widetilde{A}$$which implies that $F\cap \widetilde{A}=\mathsf{O}$ as required.
	\end{proof}
	We show that the equivalences (1), (2) and (5) of Theorem \ref{remote subbilocale} hold for bispaces, by reasoning as in the localic proof.
	\begin{theorem}\label{remote subset}
		Let $(X,\tau_{1},\tau_{2})$ be a bispace and $A\subseteq X$. The following statements are equivalent.
		\begin{enumerate} 
			\item $A$ is $(\tau_{i},\tau_{j})$-remote.
			\item $A\cap \cl_{\tau_{i}}(N)=\emptyset$ for every $(\tau_{i},\tau_{j})$-nowhere dense $N\subseteq X$.
			\item $A\subseteq U$ for every $\tau_{j}$-dense $U\in \tau_{i}$.
		\end{enumerate}
	\end{theorem}
	
	
	We move to normed lattices.. We refer the reader to \cite{AFG} for a theory of normed lattices used below. Recall that  for a linear space $E$, the pair $(E,q)$, where $q$ is a quasi-norm defined on $E$, is called the \textit{quasi-normed space}  associated to $E$, and the function $q^{-1}(a)=q(-a)$ defines a quasi-norm called a \textit{conjugate} of $q$ in $E$ and its induced quasi-uniformity $q^{-1}$ is the conjugate of $q$.

	For a normed lattice $(E,\|\;\|,\leq)$, we will consider the topologies deduced from the norm $\|\;\|$ and from the associated quasi-norm. A subset $U$ of $E$ is said to be \textit{open} (resp. dense) if it is open (resp. dense) for the norm. It is $q$-open (resp. $q$-dense) if it is open (resp. dense) for the associated quasi-norm. The topology deduced from $q^{-1}$ is coarser than that deduced from the norm.
	
	We introduce the following concept of remoteness in the category of normed lattices.
	\begin{definition}
		Let $(E,\|\;\|,\leq)$ be a normed lattice. Call $A\subseteq E$ \textit{remote} in case $A\subseteq U$ for every dense, open and decreasing $U\subseteq E$. 
	\end{definition}
	We recall the following result from \cite{AFG}.
	\begin{lemma}\label{lem}
		Let $(E,\|\;\|,\leq)$ be a normed lattice and $U\subseteq X$. Then
		\begin{enumerate}
			\item  $U$ is $q$-open iff $U$ is open and decreasing.
			\item  If $U$ is $q^{-1}$-dense and decreasing, then $U$ is dense.
		\end{enumerate}
	\end{lemma}
	Finally, we characterize remote subsets of normed lattices.
	\begin{proposition}
		Let $(E,\|\;\|,\leq)$ be a normed lattice. A subset $A$ of $E$ is remote iff it is $(q,q^{-1})$-remote.
	\end{proposition}
	\begin{proof}
		Let $U\subseteq E$ be $q$-open and $q^{-1}$-dense. It follows from Lemma \ref{lem} that $U$ is decreasing, open and dense in the norm. Since $E$ is remote, $A\subseteq U$, as required.
		
		Conversely, let $A\subseteq E$ be $(q,q^{-1})$-remote and choose a dense, open and decreasing $G\subseteq E$. It follows from Lemma \ref{lem}(1) that $G$ is $q$-open. Since the topology deduced from $q^{-1}$ is coarser than that deduced from the norm, $G$ is $q^{-1}$-dense. By hypothesis, $A\subseteq G$. Thus $A$ is remote.
	\end{proof}
	
	
	
	\section{$(i,j)$-remote sublocales and bilocalic maps}
	
	This section discusses preservation and reflection of $(i,j)$-remote and weakly $(i,j)$-remote sublocales by bilocalic maps.
	
	\begin{definition}
		We call $f:(L,L_{1},L_{2})\rightarrow (M,M_{1},M_{2})$ a \textit{bilocalic map} if (i) $f:L\rightarrow M$ is a localic map, and (ii) $f[L_{i}]\subseteq M_{i}$ and $f^{*}[M_{i}]\subseteq L_{i}$ for $i=1,2$.
		
		
	\end{definition}  
	
	For a bilocalic map $f:(L,L_{1},L_{2})\rightarrow (M,M_{1},M_{2})$, the localic map $f:L\rightarrow M$ is called the \textit{total part} of $f:(L,L_{1},L_{2})\rightarrow (M,M_{1},M_{2})$ and $f[-]:\mathcal{S}(L)\rightarrow\mathcal{S}(M)$ and $f_{-1}[-]:\mathcal{S}(M)\rightarrow\mathcal{S}(L)$ are respectively the usual localic image and localic preimage functions induced by the total part of $f$.
	

	\begin{example}
		(1) For a locale $L$ and $S\in \mathcal{S}(L)$, the inclusion map $(S,S,S)\hookrightarrow (L,L,L)$ is a bilocalic map.
		
		(2) It is clear from the above definition of a bilocalic map that, given any bilocalic map $f:(L,L_{1},L_{2})\rightarrow(M,M_{1},M_{2})$, the map $f^{*}:(M,M_{1},M_{2})\rightarrow(L,L_{1},L_{2})$ whose total part is the right adjoint of $f$, is a biframe map.  
	\end{example}
	We consider preservation of $(i,j)$-remote sublocales and weakly $(i,j)$-remote sublocales. We only prove the case of $(i,j)$-remoteness. The proof for weakly $(i,j)$-remoteness follows a similar sketch and uses the facts that the left adjoint $f^{*}$ of a localic map $f$ preserves complements and the localic pre-image function $f_{-1}[-]$ under $f$ preserves clopen sublocales.
	\begin{proposition}\label{ijndpreserve}
		Let $f:(L,L_{1},L_{2})\rightarrow(M,M_{1},M_{2})$ be a bilocalic map. Consider the following statements: 
		\begin{enumerate}
			\item $f[-]:\mathcal{S}(L)\rightarrow\mathcal{S}(M)$ preserves $(i,j)$-remote (resp. weakly $(i,j)$-remote) sublocales.
			\item $f_{-1}[-]:\mathcal{S}(M)\rightarrow\mathcal{S}(L)$ preserves $(i,j)$-nowhere dense (resp. clopen nowhere dense) sublocales.
			\item $f^{*}$ preserves $j$-dense (resp. complemented (in $L$) $j$-dense) elements.
		\end{enumerate}Then for $i\neq j\in\{1,2\}$, $(3)\Leftrightarrow (2)\Rightarrow(1)$. Moreover, if $(L,L_{1},L_{2})$ is balanced, then $(1)\Leftrightarrow (2)\Leftrightarrow(3)$.
	\end{proposition}
	\begin{proof} 
		
		$(2)\Rightarrow(3)$: Let $a\in M_{i}$ be $j$-dense. By Corollary \ref{nd elements in bilocale}, $\mathfrak{c}(a)$ is $(i,j)$-nowhere dense. By hypothesis, $f_{-1}[\mathfrak{c}(a)]$ is $(i,j)$-nowhere dense. But $f_{-1}[\mathfrak{c}(a)]=\mathfrak{o}(f^{*}(a))$, so $\mathfrak{c}(f^{*}(a))$ is $(i,j)$-nowhere dense, making $f^{*}(a)\in L_{i}$ $j$-dense by Corollary \ref{nd elements in bilocale}.
		
		$(3)\Rightarrow(2)$: Let $A\in\mathcal{S}(M)$ be $(i,j)$-nowhere dense. Then $\bigvee\{x\in M_{i}:A\subseteq \mathfrak{c}(x)\}\in M_{i}$ is $j$-dense. It follows that $f^{*}\big(\bigvee\{x\in M_{i}:A\subseteq \mathfrak{c}(x)\}\big)$ is $j$-dense and $f^{*}\big(\bigvee\{x\in M_{i}:A\subseteq \mathfrak{c}(x)\}\big)\in L_{i}$ because $f^{*}$ is a biframe homomorphism. Therefore $$\mathfrak{c}\left(f^{*}\left(\bigvee\{x\in M_{i}:A\subseteq \mathfrak{c}(x)\}\right)\right)=f_{-1}\left[\mathfrak{c}\left(\bigvee\{x\in M_{i}:A\subseteq \mathfrak{c}(x)\}\right)\right]=f_{-1}[\cl_{i}(A)]$$ is $(i,j)$-nowhere dense by Proposition \ref{nd elements in bilocale}. Because $f_{-1}[A]\subseteq f_{-1}[\cl_{i}(A)]$, it follows from Lemma \ref{ijndsubset} that $f_{-1}[A]$ is $(i,j)$-nowhere dense.

		$(2)\Rightarrow(1)$: Let $A\in\mathcal{S}(L)$ be $(i,j)$-remote and choose an $(i,j)$-nowhere dense sublocale $N$ of $M$. Set $\cl_{i}(N)=\mathfrak{c}(a)$ for some $a\in M_{i}$. By (2), $$\cl_{i}(f_{-1}[\mathfrak{c}(a)])\cap A=\cl_{i}(\mathfrak{c}(f^{*}(a)))\cap A=\mathsf{O}.$$ But $f^{*}(a)\in L_{i}$, so $\mathfrak{c}(f^{*}(a))\cap A=\mathsf{O}$. Clearly $\mathfrak{c}(a)\cap f[A]=\mathsf{O}$. Thus $\cl_{i}(N)\cap f[A]=\mathsf{O}$.
		
		For the special case, we prove $(1)\Rightarrow (3)$. Assume that $(L,L_{1},L_{2})$ is balanced, $f[-]$ preserves $(i,j)$-remote sublocales and let $a\in M_{i}$ be $j$-dense. It follows from Example \ref{exaBLij}(6) that $\mathfrak{B}L$ is $(i,j)$-remote. Since $(i,j)$-remote sublocales are contained in every open sublocale induced by $L_{j}$-elements and $f[-]$ preserves $(i,j)$-remote sublocales, $f[\mathfrak{B}L]\subseteq \mathfrak{o}(a)$. Therefore $\mathfrak{B}L\subseteq f_{-1}[\mathfrak{o}(a)]=\mathfrak{o}(f^{*}(a))$, making the $L_{i}$-element $f^{*}(a)$ a dense element of $L$ and hence $j$-dense by Observation \ref{idenseobs}.
	\end{proof}
	In the next result, we consider reflection of $(i,j)$-remoteness.
	
	\begin{proposition}\label{preimagebi}
		Let $f:(L,L_{1},L_{2})\rightarrow(M,M_{1},M_{2})$ be a bilocalic map that sends $j$-dense elements to $j$-dense elements. Then $f_{-1}[-]$ preserves $(i,j)$-remote sublocales.
	\end{proposition}
	\begin{proof}
		Let $A\in \mathcal{S}(M)$ be $(i,j)$-remote and choose a $j$-dense $x\in L_{i}$. Since $f[L_{i}]\subseteq M_{i}$ and $f$ sends $j$-dense elements to $j$-dense elements, $f(x)\in M_{i}$ is $j$-dense. Therefore $A\cap \mathfrak{c}(f(x))=\mathsf{O}$ which implies that $f_{-1}[A]\cap\mathfrak{c}(x)=\mathsf{O}$, as required.
	\end{proof}
For the case of weakly $(i,j)$-remoteness, we impose a condition on $f$ such that its total part is a lattice homomorphism. The proof is similar to that of Proposition \ref{preimagebi}.
\begin{proposition}\label{preimagebiweakly}
	Let $f:(L,L_{1},L_{2})\rightarrow(M,M_{1},M_{2})$ be a bilocalic map that sends $j$-dense elements to $j$-dense elements and such that its total part is a lattice homomorphism. Then $f_{-1}[-]$ preserves $(i,j)$-remote sublocales.
\end{proposition}

\section{The sublocale $\Rmt_{(i,j)}\!L$}
	
	For a bilocale $(L,L_{1},L_{2})$, set \[\Rmt_{(L_{i},L_{j})} (L,L_{1},L_{2})=\{x\in L:\mathfrak{c}(x)\text{ is weakly } (i,j)\text{-remote}\}.\]
	From Proposition \ref{bidense}, we get that  \[\Rmt_{(L_{i},L_{j})} (L,L_{1},L_{2})=\{x\in L:x\vee a=1\text{ for each complemented }j\text{-dense }a\in L_{i}\}.\]
	
	We shall sometimes write $\Rmt_{(i,j)}\!L$ if the bilocale $(L,L_{1},L_{2})$ is clear from the context. 
	
	For the rest of this section, we investigate some properties of $\Rmt_{(i,j)}\!L$.
	
	We consider some examples.
	
	\begin{example}\label{booleanbi}
		
		(1) For the bitopological space given in Example \ref{BLnotnd}, we have $\Rmt_{(1,2)}\!\tau=\mathfrak{c}(\{a\})$. 
		
		
		(2) A bilocale is \textit{Boolean}, \cite{S}, if $x\prec_{i}x$ for each $x\in L_{i}$, $i=1,2$, i.e., there is $c\in L_{j}$ ($i\neq j$) such that $x\wedge c=0$ and $x\vee c=1$. This tells us that each $x\in L_{i}$, $i=1,2$, is complemented in $L$. In every Boolean bilocale $(L,L_{1},L_{2})$, $L=\Rmt_{(i,j)}\!L$.
		
		(3) Consider the \textit{biframe of reals} \cite{FS} which is the triple $(\mathfrak{O}\mathbb{R},\mathfrak{O}\mathcal{D}\mathbb{R},\mathfrak{O}\mathcal{U}\mathbb{R})$, where $$\mathfrak{O}\mathcal{D}\mathbb{R}=\{(-\infty,x):x\in \mathbb{R}\}\cup\{\emptyset,\mathbb{R}\}$$ and $$\mathfrak{O}\mathcal{U}\mathbb{R}=\{(y,\infty):y\in \mathbb{R}\}\cup\{\emptyset,\mathbb{R}\}.$$ Clearly, $\Rmt_{(\mathfrak{O}\mathcal{D}\mathbb{R},\mathfrak{O}\mathcal{U}\mathbb{R})}\!\mathfrak{O}\mathbb{R}=\mathfrak{O}\mathbb{R}$.
		
		(4) For a symmetric bilocale $(L,L_{1},L_{2})$, $\Rmt_{(i,j)}\!L=L$ if and only if $L$ is Boolean. See \cite[Corollary 2.21]{N} for the proof.
		
		(5) For a bilocale $(L,L_{1},L_{2})$, if $a\leq x$ and $a\in \Rmt_{(i,j)}\!L$, then $x\in \Rmt_{(i,j)}\!L$. In particular, $1\in \Rmt_{(i,j)}\!L$.
		
		(6) Recall that a frame homomorphism $h:M\rightarrow L$ is closed if $f(a\vee h(b))=f(a)\vee b$ for every $a\in L$, $b\in M$. We also recall from \cite[Proposition 4.3.]{DM} that if an element $a\in L$ is dense in $L$, then $a\oplus1$ is dense in $L\oplus M$: Using the fact that $(x\oplus y)^{*}=(x^{*}\oplus1)\vee (1\oplus y^{*})$ for all $x\in L$, $y\in M$ (from \cite{BP0}), we get that $$(a\oplus1)^{*}=(a^{*}\oplus1)\vee(1\oplus 1^{*})=(0\oplus1)\vee (1\oplus0)=0_{L\oplus M}.$$
		Now, for any symmetric bilocales $L$ and $M$ such that the  coproduct injection
			\begin{tikzcd}
				L \arrow[r, "q_{L}"]
				&L\oplus M
			\end{tikzcd} 
			is closed, if $x\oplus y\in\Rmt_{(i,j)}\!(L\oplus M)$ for any $x\in L$ and $y\in M$, then $x\in\Rmt_{(i,j)}\!L$: Let $x\in L$, $y\in M$ and assume that $x\oplus y\in\Rmt_{(i,j)}\!(L\oplus M)$. For any dense $a\in L$, we get that $q_{L}(a)=a\oplus 1$ is dense in $L\oplus M$. Since $x\oplus y\in\Rmt_{(i,j)}\!(L\oplus M)$, we get that $(x\oplus y)\vee (a\oplus 1)=1_{L\oplus M}$, which implies that $(x\oplus y)\vee q_{L}(a)=1_{L\oplus M}$. Therefore in light of $q_{L}$ being a closed homomorphism, $(q_{L})_{*}(x\oplus y)\vee a=1$. If $a=1$, then we are done. If $a\neq 1$, then \[(q_{L})_{*}(x\oplus y)=\bigvee\{b\in L:q_{L}(b)\leq x\oplus y\}=\bigvee\{b\in L:b\oplus 1\leq x\oplus y\}\neq 0.\] Therefore there exists $b\neq 0$ in $L$ such that $0_{L\oplus M}\neq b\oplus 1\leq x\oplus y$, which implies that $b\leq x$ and $y=1$. Therefore $$1=(q_{L})_{*}(x\oplus y)\vee a=(q_{L})_{*}(x\oplus 1)\vee a=(q_{L})_{*}(q_{L}(x))\vee a=x\vee a,$$where the latter equality holds since $q_{L}$ is injective. Thus $x\in \Rmt_{(i,j)}\!L$. 
			
			Consequently, if \begin{tikzcd}
				L \arrow[r, "q_{L}"]
				& L\oplus M &M\arrow[l,swap,"q_{M}"]
			\end{tikzcd} are closed coproduct injections, then for any $x\in L$ and $y\in M$, $x\oplus y\in\Rmt_{(i,j)}\!(L\oplus M)$ implies $x\in\Rmt_{(i,j)}\!L$ and $y\in\Rmt_{(i,j)}\!M$.
	\end{example}

We prove that $\Rmt_{(i,j)}\!L$ is a sublocale and is always closed.
	\begin{proposition}\label{remclosed}
		Let $(L,L_{1},L_{2})$ be a bilocale. Then $\Rmt_{(i,j)}\!L$ is a closed sublocale of $L$.
	\end{proposition}
\begin{proof}
	Choose $\{a_{\alpha}:\alpha\in \Lambda\}\subseteq\Rmt_{(i,j)}\!L $ and let $y\in L_{i}$ be $j$-dense. Then $$y\vee \bigwedge{a_{\alpha}}=\bigwedge(y\vee a_{\alpha})=\bigwedge(1_{L})=1_{L}$$ where the first equality follows since complemented elements are colinear and the second equality follows since members of $\Rmt_{(i,j)}\!L$ meet $j$-dense elements of $L_{i}$ at the top. Therefore $\bigwedge a_{\alpha}\in\Rmt_{(i,j)}\!L$.

	Furthermore, let $x\in L$, $a\in \Rmt_{(i,j)} L$ and $y\in L_{i}$ be $j$-dense . Then $y\vee a=1_{L}$. Since $a\leq x\rightarrow a$, we have that $y\vee (x\rightarrow a)=1_{L}$. 
	Thus $x\rightarrow a\in\Rmt_{(i,j)} L$.

	Hence $\Rmt_{(i,j)}\!L$ is a sublocale of $L$.

		For closedness, we show that for every $A\in\mathcal{S}(L)$, $A\subseteq\Rmt_{(i,j)}\!L$ implies $\overline{A}\subseteq\Rmt_{(i,j)}\!L$. Assume that $A\subseteq\Rmt_{(i,j)}\!L$ and let $x\in \overline{A}$ and $y\in L_{i}$ be complemented $j$-dense. Since $\bigwedge A\in \Rmt_{(i,j)}\!L$, $\left(\bigwedge A\right)\vee y=1$. But $\bigwedge A\leq x$, so $x\vee y=1$. Thus $x\in\Rmt_{(i,j)}\!L$. Consequently, $\overline{\Rmt_{(i,j)}\!L}\subseteq\Rmt_{(i,j)}\!L$, making $\Rmt_{(i,j)}\!L$ a closed sublocale.
\end{proof}

\begin{obs}
	Since $\mathfrak{B}L$ is seldomly complemented, Proposition \ref{remclosed} tells us that $\Rmt_{(i,j)}\!L$ is not always the same as $\mathfrak{B}L$. This is also confirmed by Example \ref{booleanbi}(1) where $\mathfrak{B}\tau=\{\emptyset,X,\{a\},\{b,c\}\}\neq \Rmt_{(i,j)}\!\tau$. We also note from the recalled example that $\Rmt_{(1,2)}\!\tau$ is not always a remote sublocale of $L$. This is because $\Rmt_{(1,2)}\!\tau$ does not miss the $\tau$-nowhere dense sublocale $\mathfrak{c}(\{a,b\})$.
\end{obs}

In what follows, we consider some conditions on a bilocale $(L,L_{1},L_{2})$ such that $\Rmt_{(i,j)}\!L$ is remote.

\begin{proposition}\label{remoteRemB}
	Let $(L,L_{1},L_{2})$ be a bilocale. If $L_{i}= L$, then $\Rmt_{(i,j)}\!L$ is a remote sublocale of $L$.
\end{proposition}
\begin{proof}
	
	Let $y\in L$ be dense. Since $y\in L_{i}$, it follows from Observation \ref{idenseobs} that $y$ is $j$-dense. By Theorem \ref{remote subbilocale}, $\mathfrak{c}(x)\subseteq\mathfrak{o}(y)$ for every $x\in \Rmt_{(i,j)}\!L$ which means that $x\in \mathfrak{o}(y)$ for all $x\in \Rmt_{(i,j)}\!L$. Therefore $\Rmt_{(i,j)}\!L\subseteq\mathfrak{o}(y)$. Thus $\Rmt_{(i,j)}\!L$ is a remote sublocale of $L$.
\end{proof}
\begin{obs}
	(1) The converse of Proposition \ref{remoteRemB} is not always true. Example \ref{booleanbi}(3) is a counterexample. This tells us that the condition hypothesized in Proposition \ref{remoteRemB} is one of the many conditions making $\Rmt_{(i,j)}\!L$ remote.
	
	(2) Since $\mathfrak{B}L$ is the largest remote sublocale of a locale $L$, Theorem \ref{remoteRemB} also gives a condition when $\Rmt_{(i,j)}\!L\subseteq\mathfrak{B}L$.
\end{obs}

Although $\Rmt_{(i,j)}\!L$ may not always be remote, it is always $(i,j)$-remote, as shown below.

\begin{proposition}\label{ijremote}
	Let $(L,L_{1},L_{2})$ be a bilocale. Then $\Rmt_{(i,j)}\!L$ is $(i,j)$-remote.
\end{proposition}
\begin{proof}
	Let $x\in L_{i}$ be $j$-dense. Since $\bigwedge\!\Rmt_{(i,j)}\!L\in \Rmt_{(i,j)}\!L$, $x\vee \bigwedge\!\Rmt_{(i,j)}\!L=1$ so that $\Rmt_{(i,j)}\!L=\mathfrak{c}(\bigwedge\!\Rmt_{(i,j)}\!L)\subseteq\mathfrak{o}(x)$. 
\end{proof}

In the following result, we give neccessary and sufficient conditions for $\Rmt_{(i,j)}\!L$ to be the whole locale.

\begin{proposition}\label{remequalL}
	Let $(L,L_{1},L_{2})$ be a bilocale. The following statements are equivalent.
	\begin{enumerate}
		\item $\Rmt_{(i,j)}\!L=L$.
		\item $L$ is $(i,j)$-remote as a sublocale of itself.
		\item $1$ is the only $j$-dense element of $L_{i}$ which is complemented in $L$.
		\item $0\in\Rmt_{(i,j)}\!L$.
	\end{enumerate}
\end{proposition}
\begin{proof}
	$(1)\Rightarrow(2)$: Follows from Proposition \ref{ijremote}.
	
	$(2)\Rightarrow(3)$: Let $x\in L_{i}$ be $j$-dense and compemented in $L$. Then by (2), $L\cap\mathfrak{c}(x)=\mathsf{O}$, making $\mathfrak{c}(x)=\mathsf{O}$. Thus $x=1$.

	$(3)\Rightarrow(4)$: Since $1$ is the only $j$-dense element of $L_{i}$ which is complemented in $L$, we have that $0\vee x=1$ for every $j$-dense $x\in L_{i}$, making $0\in \Rmt_{(i,j)}\!L$.

	$(4)\Rightarrow(1)$: Follows since $\Rmt_{(i,j)}\!L$ is closed.
\end{proof}

We consider the following example.

\begin{example}
	Recall from \cite{F0} that the collection $\mathfrak{C}L$ of all congruences on a locale $L$ form a locale. The triple $(\mathfrak{C}L,\nabla_{L},\Delta_{L})$, where $\nabla_{L}=\{\nabla_{a}:a\in L\}$ and $\Delta_{L}$ is the subframe of $\mathfrak{C}L$ generated by $\{\Delta_{a}:a\in L\}$, is a bilocale called the \textit{congruence bilocale} of $L$. By \cite{FS1}, the congruence bilocale $(\mathfrak{C}L,\nabla_{L},\Delta_{L})$ of a locale $L$ satisfies the condition that each $\nabla_{a}\in \nabla_{L}$ is complemented in $\mathfrak{C}L$ and has a complement in $\Delta_{L}$. This tells us that the only $\Delta_{L}$-dense element of $\nabla_{L}$ that is also complemented in $\mathfrak{C}L$ is $1_{\mathfrak{C}L}$. By Proposition \ref{remequalL}, $\Rmt_{(i,j)}\!\mathfrak{C}L=\mathfrak{C}L$. 
	

\end{example}

Recall from \cite{BB} that the triple $(\mathfrak{J}L,(\mathfrak{J}L)_{1},(\mathfrak{J}L)_{2})$, where $\mathfrak{J}L$ is the locale of all ideals of $L$ and $(\mathfrak{J}L)_{i}$ ($i=1,2$) is the subframe of $\mathfrak{J}L$ consisting of all ideals $J\subseteq L$ generated by $J\cap L_{i}$, is a bilocale called the \textit{ideal bilocale}. A locale is said to be \textit{Noetherian} if each of elements is compact.

We have the following result.

\begin{proposition}
In a bilocale $(L,L_{1},L_{2})$ with a Noetherian total part $L$,  $\Rmt_{(i,j)}\!L=L$ iff $\Rmt_{(i,j)}\!\mathfrak{J}L=\mathfrak{J}L$. 
\end{proposition}
\begin{proof}
	Suppose that $\Rmt_{(i,j)}\!L=L$. We show that $1_{\mathfrak{J}L}$ is the only $(\mathfrak{J}L)_{j}$-dense element of $(\mathfrak{J}L)_{i}$ that is complemented in $\mathfrak{J}L$. Let $I\in (\mathfrak{J}L)_{i}$ be $(\mathfrak{J}L)_{2}$-dense and complemented in $\mathfrak{J}L$. Since in a Noetherian locale every ideal is principal, ${\downarrow}{\bigvee\!{I}}=I$, making $\bigvee\!{I}\in I$. Therefore $\bigvee\!{J}\leq x$ for some $x\in L_{i}\cap J$, making $\bigvee\!{J}=x\in L_{i}$. The fact that $I$ is complemented in $\mathfrak{J}L$ implies that there is $J\in \mathfrak{J}L$ such that $J\cap I=\mathsf{O}_{\mathfrak{J}L}$ and $J\vee I=1_{\mathfrak{J}L}$. We get that $\mathsf{O}_{\mathfrak{J}L}=J\cap I=({\downarrow}\bigvee\!J)\cap ({\downarrow}\bigvee\!I)$ and $1_{\mathfrak{J}L}=J\vee I=({\downarrow}\bigvee\!J)\vee ({\downarrow}\bigvee\!I)$. Therefore $(\bigvee\!J)\wedge (\bigvee\!I)=0_{L}$ and $(\bigvee\!J)\vee (\bigvee\!I)=1_{L}$, making $\bigvee\!I$ complemented in $L$. For $j$-density of $\bigvee\!I$, choose $a\in L_{j}$ such that $a\wedge\bigvee\!I=0_{L}$. Then $\mathsf{O}=({\downarrow}a)\cap({\downarrow}\bigvee\!I)={\downarrow}a\cap I$. Since ${\downarrow}a\in (\mathfrak{J}L)_{j}$ and $I$ is $(\mathfrak{J}L)_{j}$-dense, ${\downarrow}a=0_{\mathfrak{J}L}$. Therefore $a=0_{L}$, making $\bigvee\!I$ $j$-dense. Since, by Proposition \ref{remequalL}, $1_{L}$ is the only $j$-dense element of $L_{i}$ which is complemented in $L$, $\bigvee\!I=1$ so that $1_{\mathfrak{J}L}={\downarrow}\bigvee\!I=I$. By Proposition \ref{remequalL}, $\Rmt_{(i,j)}\!\mathfrak{J}L=\mathfrak{J}L$.
	
	On the other hand, assume that $\Rmt_{(i,j)}\!\mathfrak{J}L=\mathfrak{J}L$ and let $x\in L_{i}$ be a $j$-dense element of $L_{i}$ which is complemented in $L$. Since $a\leq x\in L_{i}\cap {\downarrow}{x}$ for each $a\in {\downarrow}{x}$ and ${\downarrow}{x}\in \mathfrak{J}L$, ${\downarrow}{x}\in (\mathfrak{J}L)_{i}$. For $j$-density of ${\downarrow}{x}$, let $J\in (\mathfrak{J}L)_{2}$ be such that $J\wedge {\downarrow}{x}=0_{\mathfrak{J}L}$. If $a\in J$, then $a\leq b$ for some $b\in J\cap L_{j}$. But $(\bigwedge J)\wedge x=0$ and $b\leq \bigvee J$, so $b\wedge x=0$ making $b=0$ since $x$ is $j$-dense. Therefore $a=0$. Thus $J=0_{\mathfrak{J}L}$ so that ${\downarrow}{x}$ is $(\mathfrak{J}L)_{j}$-dense. Because $x$ is complemented in $L$, there is $a\in L$ such that $a\vee x=1_{L}$ and $a\wedge x=0_{L}$. Therefore $({\downarrow}{a})\vee ({\downarrow}{x})=1_{\mathfrak{J}L}$ and $({\downarrow}{a})\cap ({\downarrow}{x})=0_{\mathfrak{J}L}$. Thus ${\downarrow}{x}$ is complemented in $\mathfrak{J}L$. It follows from Proposition \ref{remequalL} that ${\downarrow}{x}=1_{\mathfrak{J}L}$. Therefore $x=1_{L}$, making $\Rmt_{(i,j)}\!L=L$ by Proposition \ref{remequalL}.
\end{proof}

 For the rest of this section, we discuss $\Rmt_{(i,j)}$ as a functor and consider some other functors associated with it. 

Denote by $\BiLoc$ the category of bilocales whose morphisms are bilocalic maps.

By a \textit{weakly closed biframe homomorphism} we mean a biframe map whose total part is weakly closed. 

\begin{proposition}\label{restriction}
	Let $(L,L_{1},L_{2}),(M,M_{1},M_{2})\in\Obj(\BiLoc)$ and $f:(L,L_{1},L_{2})\rightarrow(M,M_{1},M_{2})$ be a bilocalic map such that $f^{*}:(M,M_{1},M_{2})\rightarrow(L,L_{1},L_{2})$ is weakly closed and sends $j$-dense elements to $j$-dense elements. Then $f_{|\!\Rmt_{(i,j)}\!L}[\Rmt_{(i,j)}\!L]\subseteq \Rmt_{(i,j)}M$.
\end{proposition}
\begin{proof}
	Choose $x\in \Rmt_{(i,j)}\!L$ and let $y\in M_{i}$ be $j$-dense. Since $f^{*}$ sends $j$-dense elements to $j$-dense elements, $f^{*}(y)\in L_{i}$ is $j$-dense. Therefore $f^{*}(y)\vee x=1$. But $f^{*}$ is weakly closed so $y\vee f(x)=1$. Thus $f(x)\in \Rmt_{(i,j)}M$.
\end{proof}
The following corollary is an immediate result of Proposition \ref{restriction}.
\begin{corollary}\label{restrictioncor}
	Let $(L,L_{1},L_{2}),(M,M_{1},M_{2})\in\Obj(\BiLoc)$ and $f:(L,L_{1},L_{2})\rightarrow(M,M_{1},M_{2})$ be a bilocalic map such that $f^{*}:(M,M_{1},M_{2})\rightarrow(L,L_{1},L_{2})$ is weakly closed and sends $j$-dense elements to $j$-dense elements. Then $$f_{|\!\Rmt_{(i,j)}\!L}:\Rmt_{(i,j)}\!L\rightarrow \Rmt_{(i,j)}M$$ is a localic map.
\end{corollary}

\begin{definition}
	Call a bilocalic map $f:(L,L_{1},L_{2})\rightarrow(M,M_{1},M_{2})$ a \textit{$\Rem_{(i,j)}$-map} if $f[\Rmt_{(i,j)}\!L]\subseteq \Rmt_{(i,j)}\!M$.
\end{definition}

\begin{example}
	The bilocalic maps described in the statement of Corollary \ref{restrictioncor} are $\Rem_{(i,j)}$-maps.
\end{example}



Denote by $\BiLoc_{R(i,j)}$ the subcategory of $\BiLoc$ whose morphisms are $\Rem_{(i,j)}$-maps.

There is a functor between $\BiLoc_{R(i,j)}$ and $\Loc$, as one checks routinely.

\begin{theorem}\label{functorRem}
	The assignment \[\Rmt_{(i,j)}:\BiLoc_{R(i,j)}\rightarrow\Loc,\]
	\[(L,L_{1},L_{2})\mapsto\Rmt_{(i,j)}\!L,\]\[\Rmt_{(i,j)}(f:(L,L_{1},L_{2})\rightarrow(M,M_{1},M_{2}))=f_{|\!\Rmt_{(i,j)}\!L}\]
	is a functor.
\end{theorem}


Recall from \cite{BBH} that there is a faithful functor $U:\BiFrm\rightarrow\Frm$ which takes the total part. It is clear that there is also a faithful functor $F:\BiLoc\rightarrow \Loc$ behaving the same as the functor $U$. Consider the inclusion functor $E_{\BiLoc_{R(i,j)}}:\BiLoc_{R(i,j)}\hookrightarrow \BiLoc$. One can easily see that the functor $G=F\circ E_{\BiLoc_{R(i,j)}}:\BiLoc_{R(i,j)}\rightarrow \Loc$ is faithful.

In what follows, we show that there is a natural transformation between $\Rmt_{(i,j)}$ and $G$.
\begin{proposition}\label{natural}
	There is a natural transformation $\eta:\Rmt_{(i,j)}\rightarrow G$.
\end{proposition}
\begin{proof}
	Let $(L,L_{1},L_{2})\in \BiLoc_{R(i,j)}$ and $\eta_{(L,L_{1},L_{2})}$ be the map $j_{\Rmt_{(i,j)}\!L}:\Rmt_{(i,j)}\!L\rightarrow L$. The map $\eta_{(L,L_{1},L_{2})}$ is clearly a localic map. Now, choose $f:(L,L_{1},L_{2})\rightarrow(M,M_{1},M_{2})\in \Morp(\BiLoc_{R(i,j)})$. Then the diagram 
	\begin{equation}
		\begin{tikzcd}
			{\Rmt_{(i,j)}\!L} \arrow{dd}[swap]{\Rmt_{(i,j)}(f)} \arrow{rrr}{\eta_{(L,L_{1},L_{2})}}&& & {L} \arrow{dd}{G(f)}&\\
			&& & && & \\
			{\Rmt_{(i,j)}\!M} \arrow{rrr}[swap]{\eta_{(M,M_{1},M_{2})}} && & {M} &
		\end{tikzcd}
	\end{equation}commutes: For each $x\in \Rmt_{(i,j)}\!L$, $$G(f)(\eta_{(L,L_{1},L_{2})}(x))=G(f)(x)=f(x)=\eta_{(M,M_{1},M_{2})}(f(x))=\eta_{(M,M_{1},M_{2})}(\Rmt_{(i,j)}(f)(x))$$ which proves the result.
\end{proof}


Denote by:
\begin{enumerate}
	\item $\RBiLoc_{R(i,j)}$ the full subcategory of $\BiLoc_{R(i,j)}$ whose objects are bilocales $(L,L_{1},L_{2})$ in which $\Rmt_{(i,j)}\!L$ is a remote sublocale of $L$,
	\item $\RBiLoc_{RB(i,j)}$ the full subcategory of $\RBiLoc_{R(i,j)}$ whose objects are bilocales $(L,L_{1},L_{2})$ with $1$ the only $j$-dense element of $L_{i}$, and
	\item $E_{\RBiLoc_{RB(i,j)}}$ the inclusion functor $\RBiLoc_{RB(i,j)}\subseteq\RBiLoc_{R(i,j)}$. 
\end{enumerate} 

\begin{obs}
If, in Corollary \ref{restrictioncor}, we consider the category $\RBiLoc_{R(i,j)}$ instead of $\BiLoc$, we get that $f_{|\!\Rmt_{(i,j)}\!L}$ is a closed localic map. This follows since remote sublocales are Boolean algebras and any localic map with a Boolean codomain is closed.
\end{obs}
Set  $\Rmt_{(i,j)}\circ E_{\RBiLoc_{R(i,j)}}=\Rmt_{RB}$ and $\hat{G}=G\circ E_{\RBiLoc_{RB(i,j)}}$. Since $\hat{G}$ is a composition of two faithful functors, it is faithful. We prove faithfulness of $\Rmt_{RB(i,j)}$.
\begin{proposition}\label{functorRemRB}
	The functor $\Rmt_{RB(i,j)}$ is faithful.
\end{proposition}
\begin{proof}
	Let $f,g:(L,L_{1},L_{2})\rightarrow(M,M_{1},M_{2})$ be morphisms of $\RBiLoc_{RB(i,j)}$ such that $\Rmt_{RB(i,j)}(f)=\Rmt_{RB(i,j)}(g)$. Since, by Proposition \ref{remequalL}, $\Rmt_{(i,j)}\!L=L$ and $\Rmt_{(i,j)}\!M=M$, we have that $$\text{total part of }f=f_{|\!\Rmt_{(i,j)}\!L}=\Rmt_{(i,j)}(f)=\Rmt_{(i,j)}(g)=g_{|\!\Rmt_{(i,j)}\!M}=\text{total part of }g.$$ Therefore $f=g$, making $\Rmt_{RB(i,j)}$ faithful.
\end{proof}

In fact, the functors $\Rmt_{RB(i,j)}$ and $\hat{G}$ are naturally isomorphic, as we show below.

\begin{proposition}
	The functors $\hat{G}$ and $\Rmt_{RB(i,j)}$ are naturally isomorphic.
\end{proposition}
\begin{proof}
	Consider the natural transformation $\omega:\Rmt_{RB(i,j)}\rightarrow\hat{G}$ which maps as the natural transformation $\eta:\Rmt_{(i,j)}\rightarrow G$ given in Proposition \ref{natural}. Since $L=\Rmt_{(i,j)}L=\Rmt_{RB(i,j)}L$ for every $(L,L_{1},L_{2})\in \RBiLoc_{RB(i,j)}$, each component $\omega_{(L,L_{1},L_{2})}=j_{\Rmt_{(i,j)}\!L}:\Rmt_{(i,j)}\!L\rightarrow L$ is an isomorphism. Thus $\omega$ is a natural isomorphism.
\end{proof}

We consider an endofunctor associated with $\Rmt_{(i,j)}$ that induces a comonad. We limit our focus to symmetric bilocales. Denote by $\RBiLoc_{SR(i,j)}$ the full subcategory of $\RBiLoc_{R(i,j)}$ whose objects are symmetric bilocales $(L,L_{1},L_{2})$ in which $\Rmt_{(i,j)}\!L$ is remote. 
Using the fact that $\nu_{S}[L]=S$ for every sublocale $S$ of a locale $L$, we have the following result.
\begin{proposition}
	The assignment \[\Rmt_{SB(i,j)}:\RBiLoc_{SR(i,j)}\rightarrow\RBiLoc_{SR(i,j)},\]
	\[(L,L_{1},L_{2})\mapsto(\Rmt_{(i,j)}\!L,\nu_{(\Rmt_{(i,j)}\!L)}[L_{1}],\nu_{(\Rmt_{(i,j)}\!L)}[L_{2}]),\]\[\Rmt_{SB}(f)=\Rmt_{(i,j)}(f)\]
	is an endofunctor.
\end{proposition} 
The endofunctor $\Rmt_{SB(i,j)}$ is associated with some comonad, as we show in Theorem \ref{comonad} below. We show that there are natural transformations $\eta:\Rmt_{SB(i,j)}\rightarrow \id_{\RBiLoc_{SR(i,j)}}$ and $\mu:\Rmt_{SB(i,j)}\rightarrow\Rmt_{SB(i,j)}\circ \Rmt_{SB(i,j)}$ such that the following diagrams commute:
\begin{equation}
	\begin{gathered}
		\xymatrixcolsep{6pc}\xymatrixrowsep{4pc}
		\xymatrix{
			{\Rmt_{SB(i,j)} } \ar[r]^{\mu}\ar[dr]_{\id}\ar[d]_{\mu}& {\Rmt_{SB(i,j)}\circ\Rmt_{SB(i,j)} } \ar[d]^{\eta\Rmt_{SB(i,j)}} \\
			\Rmt_{SB(i,j)}\circ\Rmt_{SB(i,j)}\ar[r]_{\Rmt_{SB(i,j)}\eta}&  {\Rmt_{SB(i,j)}}}
	\end{gathered}
\end{equation}
\begin{equation}
	\begin{gathered}
		\xymatrixcolsep{6pc}\xymatrixrowsep{4pc}
		\xymatrixcolsep{6pc}\xymatrixrowsep{4pc}\xymatrix{
			{\Rmt_{SB(i,j)} } \ar[r]^{\mu}\ar[d]_{\mu}& {\Rmt_{SB(i,j)}\circ\Rmt_{SB(i,j)} } \ar[d]^{\Rmt_{SB(i,j)}\mu} \\
			\Rmt_{SB(i,j)}\circ\Rmt_{SB(i,j)}\ar[r]_{\mu\Rmt_{SB(i,j)}}&  {\Rmt_{SB(i,j)}\circ\Rmt_{SB(i,j)}\circ\Rmt_{SB(i,j)}}}
	\end{gathered}
\end{equation}
\begin{theorem}\label{comonad}
	The triple $(\Rmt_{SB(i,j)},\eta,\mu)$, where 
	\begin{enumerate}
		\item $\Rmt_{SB(i,j)}$ is the endofunctor $\Rmt_{SB(i,j)}:\RBiLoc_{SR(i,j)}\rightarrow\RBiLoc_{SR(i,j)}$,
		\item $\eta:\Rmt_{SB(i,j)}\rightarrow \id_{\BCFLoc_{SR(i,j)}}$ is a function that assigns to each $(L,L_{1},L_{2})\in \RBiLoc_{SR(i,j)}$ the map $$\eta_{(L,L_{1},L_{2})}=j_{\Rmt_{(i,j)}\!L}:(\Rmt_{(i,j)}\!L,\nu_{(\Rmt_{(i,j)}\!L)}[L_{1}],\nu_{(\Rmt_{(i,j)}\!L)}[L_{2}])\rightarrow (L,L_{1},L_{2}),$$ and
		\item $\mu:\Rmt_{SB}\rightarrow\Rmt_{SB}\circ \Rmt_{SB}$ assigns to each $(L,L_{1},L_{2})\in \RBiLoc_{SR(i,j)}$ the map $$\mu_{(L,L_{1},L_{2})}=\id_{\Rmt_{SB}(L)}=\id_{\Rmt_{(i,j)}\!L}:\Rmt_{SB}(L,L_{1},L_{2})\rightarrow \Rmt_{SB}(\Rmt_{SB} (L,L_{1},L_{2})),$$
	\end{enumerate}  is a comonad. 
\end{theorem}
\begin{proof} Straightforward.
\end{proof}

Let $\BBooLoc_{S(i,j)}$ denote the full subcategory of $\BiLoc$ whose objects are symmetric Boolean bilocales. We close this section with a result showing that $\BBooLoc_{S(i,j)}$ is a coreflective subcategory of $\BCFLoc_{SR(i,j)}$.

Since $\Rmt_{(i,j)}L=L$ for every symmetric Boolean bilocale $(L,L_{1},L_{2})$, each bilocalic map between symmetric Boolean bilocales is a $\Rmt_{(i,j)}$-map and hence $\BBooLoc_{S(i,j)}$ is a full subcategory of $\RBiLoc_{SR(i,j)}$ because every Boolean bilocale is an object of $\RBiLoc$.

\begin{proposition}\label{reflective}
	$\BBooLoc_{S(i,j)}$ is a coreflective subcategory of $\RBiLoc_{SR(i,j)}$.
\end{proposition}
\begin{proof}
	Let $(L,L_{1},L_{2})\in \RBiLoc_{SR(i,j)}$. Since $\Rmt_{(i,j)}\!L$ is Boolean, it is clear that the triple $(\Rmt_{(i,j)}\!L,\Rmt_{(i,j)}\!L,\Rmt_{(i,j)}\!L)$ is a Boolean bilocale, making it an object of the category $\BBooLoc_{S(i,j)}$.
	
	The map $j_{\Rmt_{(i,j)}\!L}:(\Rmt_{(i,j)}\!L,\Rmt_{(i,j)}\!L,\Rmt_{(i,j)}\!L)\rightarrow (L,L_{1},L_{2})$ is a $\Rem_{(i,j)}$-map: $$(j_{\Rmt_{(i,j)}\!L})_{|\!\Rmt_{(i,j)}\!L}[\Rmt_{(i,j)}(\Rmt_{(i,j)}\!L)]=j_{\Rmt_{(i,j)}\!L}[\Rmt_{(i,j)}(\Rmt_{(i,j)}\!L)]\subseteq\Rmt_{(i,j)}\!L$$ so that $j_{\Rmt_{(i,j)}\!L}$ is a $\Rem_{(i,j)}$-map.
	
	Let $f:(N,N_{1},N_{2})\rightarrow (L,L_{1},L_{2})$ be a $\Rem_{(i,j)}$-map where $(N,N_{1},N_{2})$ is a symmetric Boolean bilocale. Then $N=\Rmt_{(i,j)}(N)$ by Example \ref{booleanbi}(4). So, there is a bilocalic map, say $f':(N,N_{1},N_{2})\rightarrow(\Rmt_{(i,j)}L,\Rmt_{(i,j)}L,\Rmt_{(i,j)}L)$, which maps as $f$. This bilocalic map clearly satisfies $c\circ f'=f$.
	
	$f'$ is unique: Let $k:(N,N_{1},N_{2})\rightarrow(\Rmt_{(i,j)}L,\Rmt_{(i,j)}L,\Rmt_{(i,j)}L)$ be a bilocalic map such that $j_{\Rmt_{(i,j)}\!L}\circ k=f$. Then, for each $x\in N$, $k(x)=j_{\Rmt_{(i,j)}\!L}(k(x))=f(x)=c(f'(x))=f'(x)$. Thus $k=f'$ so that $f'$ is unique.
\end{proof}

	
	\section*{Acknowledgement}
	The author expresses his gratitude to his Ph.D. supervisor Professor Themba Dube who guided and supported him throughout the preparation stages of this article. He acknowledges funding from the National Research Foundation of South Africa under Grant 134159 and passes words of thanks to the referee for the valuable comments.
	

\end{document}